\documentclass[11pt]{amsart}
\usepackage{amssymb}
\usepackage[hyphens]{url} 
\usepackage[dvips]{graphicx}
\usepackage{amsmath}
\usepackage{cite}
\usepackage{mathrsfs}
\usepackage[title]{appendix} 
\usepackage[utf8]{inputenc}
\usepackage[english]{babel}
\usepackage{hyperref}
\RequirePackage{textgreek}
\usepackage[all]{xy}
\usepackage{enumitem}
\usepackage{bm}

\hypersetup{linktoc=page}
 \theoremstyle{plain}
 \newtheorem{thm}{\bfseries Theorem}[section]
 \newtheorem{theorem}{\bfseries Theorem}
 \newtheorem{prop}[thm]{\bfseries Proposition}

 \newtheorem{question}{\bfseries Question}[section]
 \newtheorem{problem}{\bfseries Problem}[section]
 \newtheorem{lem}[thm]{\bfseries Lemma}
 
 \newtheorem{cor}[thm]{\bfseries Corollary}
 
 \newtheorem{dfn}[thm]{\bfseries Definition}
 \theoremstyle{remark}
 \newtheorem{rem}[thm]{Remark}
 \numberwithin{equation}{thm}
 \theoremstyle{lemmastyle}
 \newtheorem*{lem*}{Lemma}

  \renewcommand{\leq}{\leqslant}
\renewcommand{\geq}{\geqslant}

\newcommand{\spa}{\overset{{\scriptscriptstyle \mathrm{sp}}}{\curvearrowright}}
\title[Poissonian Actions of Polish Groups]{Poissonian Actions of Polish Groups}



\author[Avraham-Re'em \& Roy]{\bfseries Nachi Avraham-Re'em \& Emmanuel Roy}

\address{Department of Mathematical Sciences, Chalmers University of Technology and the University of Gothenburg, Gothenburg, Sweden}

\email{nachman@chalmers.se}

\address{Laboratoire Analyse Géométrie et Applications, Université Paris 13, Institut Galilée, 99 avenue Jean-Baptiste Clément, 93430 Villetaneuse, France}

\email{roy@math.univ-paris13.fr}

\thanks{The research of N. A-R was supported by ISF (grant No. 1180/22) and the Knut and Alice Wallenberg Foundation (KAW 2021.0258).}

\begin{document}

\begin{abstract}
We define and study Poissonian actions of Polish groups as a framework to Poisson suspensions, characterize them spectrally, and provide a complete characterization of their ergodicity. We further construct {\it spatial} Poissonian actions, answering partially a question of Glasner, Tsirelson \& Weiss about L\'{e}vy groups. We also construct for every diffeomorphism group an ergodic free spatial probability preserving action. This constitutes a new class of Polish groups admitting non-essentially countable orbit equivalence relations, obtaining progress on a problem of Kechris.
\end{abstract}

\maketitle

\tableofcontents

\section{Introduction}

A well-known construction in probability theory is the Poisson point process, in which every standard (typically infinite) measure space $\left(X,\mathcal{B},\mu\right)$ is associated a standard probability space $\left(X^{\ast},\mathcal{B}^{\ast},\mu^{\ast}\right)$, whose points are configurations of points from $X$ and their distribution is governed by Poisson distribution with intensity $\mu$. In ergodic theory, one naturally associates with every measure preserving transformation $T$ of $\left(X,\mathcal{B},\mu\right)$ a probability preserving transformation $T^{\ast}$ of $\left(X^{\ast},\mathcal{B}^{\ast},\mu^{\ast}\right)$, namely the {\bf Poisson suspension}. In the first part of this work we aim to put the constructions of Poisson point process and the Poisson suspension in a general, more axiomatic framework, thus defining the notion of measure preserving {\bf Poissonian action} of Polish groups.

For a parameter $0\leq\alpha\leq\infty$, denote by $\mathrm{Poiss}\left(\alpha\right)$ the Poisson distribution with mean $\alpha$, with the convention that for $\alpha\in\left\{0,\infty\right\}$ it is the distribution of the constant $\alpha$.

\begin{dfn}[Poisson point process]
\label{Definition: Poisson point process}
Let $\left(X,\mathcal{B},\mu\right)$ be a standard measure space and $\left(\Omega,\mathcal{F},\mathbb{P}\right)$ be a standard probability space. A collection
$$\mathcal{P}=\left\{P_{A}:A\in\mathcal{B}\right\}$$
of random variables that are defined on $\left(\Omega,\mathcal{F},\mathbb{P}\right)$ is called a {\bf Poisson point process} with the {\bf base space} $\left(X,\mathcal{B},\mu\right)$, if the following properties hold:
\begin{enumerate}
\item $P_{A}$ has distribution $\mathrm{Poiss}\left(\mu\left(A\right)\right)$ for every $A\in\mathcal{B}$.
\item $P_{A\cup B}=P_{A}+P_{B}$ $\mathbb{P}$-a.s. whenever $A,B\in\mathcal{B}$ are disjoint.
\end{enumerate}
Such a Poisson point process $\mathcal{P}$ will be called {\bf generative} if, in addition,
\begin{enumerate}
    \setcounter{enumi}{2}
    \item The members of $\mathcal{P}$ generate $\mathcal{F}$ modulo $\mathbb{P}$.
\end{enumerate}
\end{dfn}

\begin{rem}
    By the famous R\'{e}nyi Theorem, which is valid in our general setting, a Poisson point process $\mathcal{P}$ as in Definition \ref{Definition: Poisson point process} automatically satisfies that $P_{A_{1}},\dotsc,P_{A_{n}}$ are independent whenever $A_{1},\dotsc,A_{n}\in\mathcal{B}$ are disjoint.
\end{rem}

The measure $\mu$ is sometimes referred to as the {\bf intensity} of $\mathcal{P}$. The classical (generative) Poisson point process with an arbitrary base space $\left(X,\mathcal{B},\mu\right)$, is usually constructed on a standard probability space $\left(X^{\ast},\mathcal{B}^{\ast},\mu^{\ast}\right)$, where $X^{\ast}$ consists of Borel simple counting measures on $X$, and it is the collection 
$$\mathcal{N}=\left\{N_{A}:A\in\mathcal{B}\right\}\text{ given by }N_{A}\left(\omega\right)=\omega\left(A\right).$$
As we shall see in Proposition \ref{Proposition: classical Poisson point process}, this construction of Poisson point process amounts to a choice of topology which is not always canonical, and $\mathcal{N}$ in its $A$-variable becomes a {\it random measure} on $X$, a property that is not assumed a priori for general Poisson point process as in Definition \ref{Definition: Poisson point process}. Nevertheless, as we shall see in Proposition \ref{Proposition: Poisson point process}, all Poisson point processes are essentially unique, and in particular all form random measures in a precise sense. Despite this universality of the Poisson point process, the ability to deviate oneself from the aforementioned concrete construction will be of great importance to us as we shall see in Theorems \ref{Theorem: Poisson random set} and \ref{Theorem: spatial Poisson suspension}.

Given a Poisson point process $\mathcal{P}$ as in Definition \ref{Definition: Poisson point process}, let $\mathcal{B}_{\mu}$ be the Borel sets in $\mathcal{B}$ with finite measure, and look at the real Hilbert space
$$H\left(\mathcal{P}\right):=\overline{\mathrm{span}}\left\{ P_{A}:A\in\mathcal{B}_{\mu}\right\}\subset L_{\mathbb{R}}^{2}\left(\Omega,\mathcal{F},\mathbb{P}\right).$$

\begin{dfn}[Poissonian action]
\label{Definition: Poissonian action}
Let $\mathcal{P}$ be a generative Poisson point process as in Definition \ref{Definition: Poisson point process}. A probability preserving action $\mathbf{S}:G\curvearrowright\left(\Omega,\mathcal{F},\mathbb{P}\right)$ of a Polish group $G$ is said to be a {\bf Poissonian action} with respect to $\mathcal{P}$, if its Koopman representation preserves $H\left(\mathcal{P}\right)$ within $L_{\mathbb{R}}^{2}\left(\Omega,\mathcal{F},\mathbb{P}\right)$.
\end{dfn}

In the next theorem we provide a characterization of Poissonian actions. We will start by introducing the natural source for Poissonian actions, namely the {\bf Poisson suspension} construction, omitting essential technical details that will be treated with great care in Proposition \ref{Proposition: classical Poisson point process}. Observe that if $T$ is a measure preserving transformation of $\left(X,\mathcal{B},\mu\right)$, one may define a probability preserving transformation $T^{\ast}$ of $\left(X^{\ast},\mathcal{B}^{\ast},\mu^{\ast}\right)$ by the property that for every $\omega$ in an appropriate $\mu^{\ast}$-conull set,
$$T^{\ast}\left(\omega\right)=\omega\circ T^{-1}.$$
Evidently, we have the property
$$N_{A}\left(T^{\ast}\left(\omega\right)\right)=N_{T^{-1}\left(A\right)}\left(\omega\right)\text{ for }A\in\mathcal{B}\text{ and for }\mu^{\ast}\text{-a.e. }\omega\in X^{\ast}.$$
This readily implies that $T^{\ast}$ is a Poissonian transformation with respect to the Poisson point process $\mathcal{N}$. As we shall see later on, this source of Poissonian actions is not limited to a single transformation but for every measure preserving action $\mathbf{T}:G\curvearrowright\left(X,\mathcal{B},\mu\right)$ of an arbitrary Polish group $G$ we obtain a Poissonian action $\mathbf{T}^{\ast}:G\curvearrowright\left(X^{\ast},\mathcal{B}^{\ast},\mu^{\ast}\right)$ with respect to the Poisson point process $\mathcal{N}$. In the Poisson suspension construction, the action $\mathbf{T}$ is referred to as a {\bf base action} of $\mathbf{T}^{\ast}$, and for a general Poissonian action this will be put as follows.

\begin{dfn}[Base of a Poissonian action]
\label{Definition: base action}
Let $\mathbf{S}:G\curvearrowright\left(\Omega,\mathcal{F},\mathbb{P}\right)$ be a Poissonian action of a Polish group $G$ with respect to a generative Poisson point process $\mathcal{P}$ as in Definition \ref{Definition: Poissonian action}. An action $\mathbf{T}:G\curvearrowright\left(X,\mathcal{B},\mu\right)$ is called a {\bf base action} for $\mathbf{S}$ if
$$P_{A}\circ S_{g}=P_{T_{g}^{-1}\left(A\right)}\quad\mathbb{P}\text{-a.e. for every }g\in G\text{ and }A\in\mathcal{B}.$$
\end{dfn}

The following theorem is our main result about Poissonian actions. We put it in a principled form so to make things clear and the precise formulations can be found in Theorems \ref{Theorem: Poissonian actions I} and \ref{Theorem: Poissonian actions II} and Corollary \ref{Corollary: Poissonian actions II}.

\begin{theorem}
\label{Theorem: Poissonian actions}
Suppose $\mathcal{P}$ is a generative Poisson point process as in Definition \ref{Definition: Poisson point process}, and that a Polish group -- the acting group -- is given. Then:
\begin{enumerate}
    \item Every measure preserving action on the base space of $\mathcal{P}$ is a base action of an essentially unique Poissonian action with respect to $\mathcal{P}$.
    \item Every Poissonian action with respect to $\mathcal{P}$ admits an essentially unique base action on the base space of $\mathcal{P}$.
\end{enumerate}
\end{theorem}

In the next we completely characterize the ergodicity of Poissonian actions in terms of their base actions, to the generality of measure preserving actions of Polish groups. For $G=\mathbb{Z}$ it was proved by Marchat \cite{marchat1978} and other proofs were given later by Grabinsky \cite[Theorem 1]{grabinsky1984poisson} and Roy \cite[$\mathsection$4.5]{roy2007ergodic} (see also Remark \ref{Remark: ergodicity} below).

\begin{theorem}
\label{Theorem: ergodicity}
Suppose $\mathbf{S}$ is a Poissonian action with a base action $\mathbf{T}$. The following are equivalent.
\begin{enumerate}
    \item $\mathbf{S}$ is ergodic.
    \item $\mathbf{S}$ is weakly mixing.
    \item $\mathbf{T}$ admits no invariant set of a positive finite measure. 
\end{enumerate}
\end{theorem}

The continuation of our study of Poissonian actions is in the more restrictive framework of {\bf spatial actions} of Polish groups. As opposed to the general notion of measure preserving actions, in which every group element corresponds to a transformation that is defined almost everywhere, in spatial actions one requires an actual Borel action which happens to admit an invariant (or quasi-invariant) measure. In many common cases, such as locally compact Polish groups, the {\it Mackey property} ensures that there is no essential difference between the two notions. However, this completely fails for general Polish groups, as was demonstrated by Becker and by Glasner, Tsirelson \& Weiss in {\it L\'{e}vy groups}. We discuss this in Section \ref{Section: Preliminaries: spatial actions and the point-realization problem}.

Observe that the usual construction of the Poisson suspension $\mathcal{N}$ fails in the spatial category. Indeed, the standard Borel space $\left(X^{\ast},\mathcal{B}^{\ast}\right)$ is classically defined as the space of simple counting Radon measures, namely Radon measures that taking nonnegative integer values and are finite on bounded sets, with respect to an appropriate choice of a metric. Then for a general Borel transformation $T$ of $\left(X,\mathcal{B}\right)$, since it may not preserve the class of bounded sets, there is no apparent way to identify $T$ with a Borel transformation of $\left(X^{\ast},\mathcal{B}^{\ast}\right)$. In order to deal with it we introduce a construction of Poisson point process as a random closed set, calling it {\bf Poisson random set}. This provides a construction of Poisson point processes in Polish topologies that may not be locally compact, and manifests the advantage of treating Poisson point processes abstractly.

Recall that if $\left(X,\mathcal{B}\right)$ is a standard Borel space, then with every Polish topology $\tau$ on $X$ that generates $\mathcal{B}$ is associated the {\bf Effros Borel space},
$$\mathbf{F}_{\tau}\left(X\right):=\left\{F\subset X:X\backslash F\in\tau\right\},$$
whose points are the $\tau$-closed sets. It is known that $\mathbf{F}_{\tau}\left(X\right)$ has a structure of a standard Borel space, that will be referred to as the {\it Effros σ-algebra} and denote by $\mathcal{E}_{\tau}\left(X\right)$. See the presentation in Section \ref{Section: Poisson random set}. Thus, a {\bf random closed set} is a probability measure on $\left(\mathbf{F}_{\tau}\left(X\right),\mathcal{E}_{\tau}\left(X\right)\right)$.

\begin{theorem}[Poisson random set]
\label{Theorem: Poisson random set}
Let $\left(X,\mathcal{B}\right)$ be a standard Borel space. For every Polish topology $\tau$ that generates $\mathcal{B}$, there are random variables 
$$\left\{\Xi_{A}:A\in\mathcal{B}\right\}\text{ of the form }\Xi_{A}:\mathbf{F}_{\tau}\left(X\right)\to\mathbb{Z}_{\geq 0}\cup\left\{ \infty\right\},$$
with the following property:
\begin{quote}
For every Borel non-atomic measure $\mu$ on $\left(X,\mathcal{B}\right)$ that is $\tau$-locally finite, there exists a unique random closed set $\mu_{\tau}$ on $\left(\mathbf{F}_{\tau}\left(X\right),\mathcal{E}_{\tau}\left(X\right)\right)$, with respect to which $\left\{\Xi_A:A\in\mathcal{B}\right\}$ forms a Poisson point process with base space $\left(X,\mathcal{B},\mu\right)$.
\end{quote}
\end{theorem}

Our first main result about spatial Poissonian actions allows one to construct spatial probability preserving actions out of spatial infinite measure preserving actions. This will be established for actions with an appropriate Polish topology in the following sense.

\begin{dfn}
A {\bf locally finite Polish action} is a measure preserving action  $\mathbf{T}:G\curvearrowright\left(X,\mathcal{B},\mu\right)$ with a Polish topology for which, simultaneously, $\mathbf{T}$ is Polish and $\mu$ is locally finite.
\end{dfn}

\begin{theorem}
\label{Theorem: spatial Poisson suspension}
Every locally finite Polish action of a Polish group is a base action of a \emph{spatial} Poissonian action.
\end{theorem}

As simple it may look, the construction of spatial Poissonian actions in Theorem \ref{Theorem: spatial Poisson suspension} provides a valuable tool to construct probability preserving spatial actions for Polish groups without appealing to the Mackey property. Our following results demonstrate the strength of this construction.

In their work on probability preserving spatial actions of L\'{e}vy groups, Glasner, Tsirelson \& Weiss showed that all such actions are necessarily trivial in that the measure must be supported on the set of fixed points \cite[Theorem 1.1]{glasnertsirelsonweiss2005}, and they left it as an open question whether L\'{e}vy groups may admit nontrivial nonsingular spatial actions \cite[Question 1.2]{glasnertsirelsonweiss2005}. Using Theorem \ref{Theorem: spatial Poisson suspension} we obtain the following partial answer:

\begin{theorem}
\label{Theorem: spatial actions}
A Polish group admits a nontrivial probability preserving spatial action if it admits any of the following nontrivial actions:
\begin{enumerate}
    \item A locally finite Polish (measure preserving) action.
    \item A locally finite Polish nonsingular action with continuous Radon-Nikodym cocycle.
\end{enumerate}
In particular, L\'{e}vy groups admit no such nontrivial actions.
\end{theorem}

An immediate strengthening of Theorem \ref{Theorem: spatial actions} can be obtained using a recent result of Kechris, Malicki, Panagiotopoulos \& Zielinski \cite[Theorem 2.3]{kechris2022polish}. Recall that an action is {\bf faithful} if each group element, except for the identity, acts nontrivially on a positive measure set.

\begin{cor}
\label{Corollary: free spatial action}
Every Polish group $G$ admitting one of the actions (1) or (2) as in Theorem \ref{Theorem: spatial actions} which is also faithful, admits a free probability preserving spatial action.
\end{cor}

We now move to use the construction of spatial Poissonian action in Theorem \ref{Theorem: spatial Poisson suspension} to construct nontrivial spatial actions in the class of diffeomorphism groups. Let $M$ be a Hausdorff connected compact finite dimensional smooth manifold. We call by a {\bf diffeomorphism group} of $M$, for some $1\leq r\leq\infty$, the group
$$\mathrm{Diff}^{r}\left(M\right)$$
of all $C^{r}$-diffeomorphisms of $M$ to itself, considered as a (non-locally compact) Polish group with the compact-open $C^{r}$-topology.

The aforementioned Mackey property, established by Mackey for locally compact groups, was generalized by Kwiatkowska \& Solecki to groups of isometries of locally compact metric spaces and, to the best of our knowledge, currently this is the largest class of Polish groups for which the Mackey property is known to hold. The following theorem shows that diffeomorphism groups are not in this class whenever $r\neq d+1$. It is the consequence of two highly nontrivial results: one is a theorem due to Herman, Thurston and Mather, following a theorem by Epstein, about the simplicity of the identity component in diffeomorphism groups, and another by Kwiatkowska \& Solecki, following Gao \& Kechris, about the topological structure of isometry groups of locally compact metric spaces.

\begin{theorem}
\label{Theorem: Diff}
For every compact smooth $d$-dimensional manifold $M$ and every $1\leq r\leq\infty$ with $r\neq d+1$, the diffeomorphism group $\mathrm{Diff}^{r}\left(M\right)$ is \emph{not} a group of isometries of a locally compact metric space.
\end{theorem}

Although the Mackey property for diffeomorphism groups is unknown, we use the construction of spatial Poissonian actions to construct spatial actions of such groups. In this context it is worth mentioning that by a result of Megrelishvili \cite[Theorem 3.1]{megrelishvili2001every} (see also \cite[Remark 1.7]{glasnertsirelsonweiss2005}), the homeomorphism group of the unit interval admits no action whatsoever. The picture in diffeomorphism groups turns out to be very different.

\begin{theorem}
\label{Theorem: Diff spatial action}
Every diffeomorphism group admits an ergodic free probability preserving spatial action. Hence, diffeomorphism groups are never L\'{e}vy.
\end{theorem}

Theorem \ref{Theorem: Diff spatial action} has a consequence in the theory of equivalence relations on standard Borel spaces. An open problem of Kechris asks whether every non-locally compact Polish group admits a non-essentially countable equivalence relation. Recently, this question was answered affirmatively for groups of isometries of a locally compact metric space by Kechris, Malicki, Panagiotopoulos \& Zielinski. While diffeomorphism groups, when $r\neq d+1$, do not belong to this class by Theorem \ref{Theorem: Diff}, we have the following corollary, which is the result of Theorem \ref{Theorem: Diff spatial action} together with a theorem by Feldman \& Ramsay:

\begin{cor}
\label{Corollary: Diff equivalence relations}
Every diffeomorphism group admits a non-essentially countable orbit equivalence relation.
\end{cor}

\subsubsection*{Acknowledgment}

We are grateful to Zemer Kosloff for early discussions that led to this work, to Michael Bj\"{o}rklund for many fruitful discussions around the subjects of this work, and to Jeff Steif for a few enlightening discussions revolving infinitely divisible distributions. We also thank Sasha Danilenko for his comments and particularly for his Remark \ref{Remark: ergodicity}. A special thank goes to Cy Maor, Klas Modin and Alexander Schmeding for sharing with us from their expertise in diffeomorphism groups.\\

\begingroup
\footnotesize
\noindent \textbf{Post-publication Note:} Subsequent to the first publication of this work, Fabien Hoareau and Fran\c{c}ois Le Ma\^{\i}tre published a work which can be found in \cite{hoareau2024spatial}, and their {\it Theorem E} is essentially the same as our Theorem \ref{Theorem: spatial actions} (the relations between the proofs were discussed there). Additionally, their {\it Question 2} is the same as our Question \ref{Question: Radon model}.
\endgroup

\section{General preliminaries}

Let $\left(X,\mathcal{B}\right)$ be a standard Borel space. Thus, $X$ is equipped with a σ-algebra $\mathcal{B}$ that is the Borel σ-algebra of some unspecified Polish topology $\tau$ on $X$. A {\bf transformation} of $\left(X,\mathcal{B}\right)$ is an invertible Borel mapping $T:X\to X$. A {\bf Borel} ($\tau$-{\bf Polish}) {\bf action} of a Polish group $G$ on $X$ is a jointly Borel (jointly $\tau$-continuous, resp.) map $\mathbf{T}:G\times X\to X$, $\mathbf{T}:\left(g,x\right)\mapsto T_{g}\left(x\right)$, such that $T_{e}=\mathrm{Id}_{X}$, where $e\in G$ is the group identity, and $T_{g}\circ T_{h}=T_{gh}$ for every $g,h\in G$. By a {\bf standard measure space} we refer to $\left(X,\mathcal{B},\mu\right)$, where $\left(X,\mathcal{B}\right)$ is a standard Borel space and $\mu$ is a measure that belongs to one of the following of classes:
\begin{description}
    \item[$\mathcal{M}_{1}\left(X,\mathcal{B}\right)$] non-atomic Borel probability measures on $X$.
    \item[$\mathcal{M}_{\sigma}\left(X,\mathcal{B}\right)$] non-atomic infinite σ-finite Borel measures on $X$.
    \item[$\mathcal{M}_{\sigma}^{\tau}\left(X,\mathcal{B}\right)$] those measures in $\mathcal{M}_{\sigma}\left(X,\mathcal{B}\right)$ that are  $\tau$-locally finite.
\end{description}
Being $\tau$-locally finite, by definition, means that every point in $X$ has a $\tau$-neighborhood of finite measure. It is equivalent to the existence of a countable base, or to the existence of a countable open cover, that consists of finite measure sets. The general theory of point processes is usually developed for Radon measure, which are nothing but locally finite measure with respect to a locally compact topology (see e.g. \cite[Theorem 7.8]{folland1999real}). Here we deal with general Polish topologies.

If $\left(X,\mathcal{B},\mu\right)$ is a standard measure space, we denote by $\mathcal{B}_{\mu}$ the ideal of sets $A\in\mathcal{B}$ for which $\mu\left(A\right)<\infty$. We use the common terminology of {\bf $\mu$-a.e.} or {\bf $\mu$-conull}, applied to a specified property of the elements of $X$, to indicate that the property holds true for all the members of some set in $\mathcal{B}$ whose complement is $\mu$-null, namely is assigned zero by $\mu$. By a {\bf transformation} of $\left(X,\mathcal{B},\mu\right)$ we refer to a bi-measurable bijective map between two $\mu$-conull sets of $X$. A transformation $T$ of $\left(X,\mathcal{B}\right)$ is said to be {\bf measure preserving} if $\mu\circ T^{-1}=\mu$, and {\bf nonsingular} if $\mu\circ T^{-1}$ and $\mu$ are in the same measure class, namely they are mutually absolutely continuous. We denote by 
$$\mathrm{Aut}\left(X,\mathcal{B},\mu\right)\text{ and }\mathrm{Aut}\left(X,\mathcal{B},\left[\mu\right]\right)$$
the groups of equivalence classes, up to equality on a $\mu$-conull set, of measure preserving and nonsingular transformations, respectively. The latter group clearly depends only on the measure class $\left[\mu\right]$ of $\mu$ rather than $\mu$ itself, and it becomes a Polish group with the weak topology, in which $S_{n}\xrightarrow[n\to\infty]{}\mathrm{Id}$ if $\mu\left(S_{n}A\triangle A\right)\xrightarrow[n\to\infty]{}$ for every $A\in\mathcal{B}$ with $\mu\left(A\right)<\infty$ and $\frac{d\mu\circ S_{n}}{d\mu}\xrightarrow[n\to\infty]{}1$ in measure. The former group then becomes a closed, hence Polish, subgroup of the latter (see e.g. \cite[end of $\mathsection$1.0]{aaronson}, \cite[Exercise (17.46)]{kechris2012classical}).

\section{Foundations of Poisson point processes}
\label{Part: General framework}

Recall Definition \ref{Definition: Poisson point process} for the general notion of Poisson point process. In the following we introduce the usual concrete construction of the Poisson point process that should be regarded as a folklore. The details of this construction will be important to us for the general context and later uses.

Let $\left(X,\mathcal{B}\right)$ be a standard Borel space. By the Isomorphism Theorem of standard Borel spaces (see e.g. \cite[$\mathsection$15.B]{kechris2012classical}) there can be found a complete metric on $X$ that induces a locally compact Polish topology $\tau$ on $X$, and thus we may relate to bounded Borel sets in $X$ with respect to a fixed choice of such a metric. We may further assume that $\tau$ admits a countable base that consists of bounded sets and, in fact, we may assume that this topology has all properties of the usual topology on $\mathbb{R}$. Denote by $X_{\tau}^{\ast}$ the space of simple counting Radon measures on $X$. That is, a Borel measure $\omega$ on $X$ is in $X_{\tau}^{\ast}$ if it satisfies the following properties:
\begin{enumerate}
    \item (Radon) $\omega\left(A\right)<\infty$ for every bounded set $A\in\mathcal{B}$.
    \item (counting) $\omega\left(A\right)\in\mathbb{Z}_{\geq 0}\cup\left\{\infty\right\}$ for every $A\in\mathcal{B}$.
    \item (simple) $\omega\left(\left\{x\right\}\right)\in\left\{0,1\right\}$ for every $x\in X$.
\end{enumerate}
The space $X_{\tau}^{\ast}$ becomes a standard Borel space with the σ-algebra $\mathcal{B}_{\tau}^{\ast}$ that is generated by the canonical mappings
\begin{equation}
\label{eq: classical Poisson}
\mathcal{N}=\left\{N_{A}:A\in\mathcal{B}\right\},\quad N_{A}:X_{\tau}^{\ast}\to\mathbb{Z}_{\geq 0}\cup\left\{\infty\right\},\quad N_{A}:\omega\mapsto\omega\left(A\right).
\end{equation}
For details on the standard Borel structure of $X_{\tau}^{\ast}$ see \cite[$\mathsection$9.1]{daley2008introduction}.

Suppose now that $\left(X,\mathcal{B},\mu\right)$ is a standard infinite measure space, that is we are given a measure $\mu\in\mathcal{M}_{\sigma}\left(X,\mathcal{B}\right)$. By the Isomorphism Theorem for standard measure spaces (see e.g. \cite[$\mathsection$17.F]{kechris2012classical}) there can be found a complete metric that induces a Polish topology $\tau$ that satisfies all of the above and, at the same time, turns $\mu$ into a Radon measure, i.e. $\mu\in\mathcal{M}_{\sigma}^{\tau}\left(X,\mathcal{B}\right)$. In fact, up to a Borel isomorphism, we may assume that $\left(X,\mathcal{B},\mu\right)$ is $\mathbb{R}$ with its usual Borel structure and the Lebesgue measure. By the classical construction of the Poisson point process, there exists a unique probability measure $\mu_{\tau}^{\ast}\in\mathcal{M}_{1}\left(X_{\tau}^{\ast},\mathcal{B}_{\tau}^{\ast}\right)$ with respect to which the random variables $\mathcal{N}$ as in \eqref{eq: classical Poisson} form a generative Poisson point process with base space $\left(X,\mathcal{B},\mu\right)$. For details about this classical construction we refer to \cite[$\mathsection$2.4]{daley2003introduction}, \cite[$\mathsection$ 3]{last2017lectures}, \cite[Proposition 19.4]{sato1999levy}. In our context we put this as follows:

\begin{prop}
\label{Proposition: classical Poisson point process}
For every standard measure space $\left(X,\mathcal{B},\mu\right)$ there exists a Polish topology $\tau$ and a standard probability space $\left(X_{\tau}^{\ast},\mathcal{B}_{\tau}^{\ast},\mu_{\tau}^{\ast}\right)$ that is defined uniquely by the property that the collection of random variables $\mathcal{N}$ as in \eqref{eq: classical Poisson} forms a generative Poisson point process with base space $\left(X,\mathcal{B},\mu\right)$. Moreover, there is a continuous embedding of Polish groups
$$\mathrm{Aut}\left(X,\mathcal{B},\mu\right)\hookrightarrow\mathrm{Aut}\left(X_{\tau}^{\ast},\mathcal{B}_{\tau}^{\ast},\mu_{\tau}^{\ast}\right),\quad T\mapsto T^{\ast},$$
such that for every $T\in\mathrm{Aut}\left(X,\mathcal{B},\mu\right)$ there is a $\mu_{\tau}^{\ast}$-conull set on which
\begin{equation}
    \label{eq: equivariance}
    N_{A}\circ T^{\ast}=N_{T^{-1}\left(A\right)}\text{ for every }A\in\mathcal{B}.
\end{equation}
\end{prop}

\begin{proof}
Thanks to the Isomorphism Theorem for standard measure spaces, we may assume that $\left(X,\mathcal{B},\mu\right)$ is nothing but the real line with its Lebesgue measure, for which the aforementioned classical construction of the Poisson point process with respect to the usual topology is well known. Let us show the second part. Pick a countable base $\mathcal{O}$ for $\tau$ consisting of $\mu$-finite measure sets. An arbitrary element $\left[T\right]_{\mu}\in\mathrm{Aut}\left(X,\mathcal{B},\mu\right)$ will be mapped to $\left[T^{\ast}\right]_{\mu_{\tau}^{\ast}}\in\mathrm{Aut}\left(X_{\tau}^{\ast},\mathcal{B}_{\tau}^{\ast},\mu_{\tau}^{\ast}\right)$ as follows. Pick a representative $T\in\left[T\right]_{\mu}$ and consider the measurable set
$$X_{\tau}^{\ast}\left(T\right):=\bigcap\nolimits_{O\in\mathcal{O}}\bigcap\nolimits_{n\in\mathbb{Z}}\left\{ N_{T^{n}\left(O\right)}<\infty\right\}.$$
By the construction of $\mu_{\tau}^{\ast}$, as $T$ preserves $\mu$ we see that $\mu_{\tau}^{\ast}\left(X_{\tau}^{\ast}\triangle X_{\tau}^{\ast}\left(T\right)\right)=0$. Let $T^{\ast}$ be the automorphism of $\left(X_{\tau}^{\ast}\left(T\right),\mathcal{B}_{\tau}^{\ast}\left(T\right),\mu_{\tau}^{\ast}\mid_{X_{\tau}^{\ast}\left(T\right)}\right)$ that is given by
$$T^{\ast}\left(\omega\right)=\omega\circ T^{-1},\quad\omega\in X_{T}^{\ast}.$$
As $\mu_{\tau}^{\ast}\left(X_{\tau}^{\ast}\triangle X_{\tau}^{\ast}\left(T\right)\right)=0$, the element $\left[T^{\ast}\right]_{\mu_{\tau}^{\ast}}\in\mathrm{Aut}\left(X_{\tau}^{\ast},\mathcal{B}_{\tau}^{\ast},\mu_{\tau}^{\ast}\right)$ is well-defined. This defines the desired mapping, that from now on we abbreviate without the equivalence class notations, i.e. $\mathrm{Aut}\left(X,\mathcal{B},\mu\right)\to\mathrm{Aut}\left(X_{\tau}^{\ast},\mathcal{B}_{\tau}^{\ast},\mu_{\tau}^{\ast}\right)$, $T\mapsto T^{\ast}$. It is clearly a homomorphism. In order to see that it is injective, note that if $T\neq\mathrm{Id}_{X}\in\mathrm{Aut}\left(X,\mathcal{B},\mu\right)$ there is a Borel set of the form $A=T^{-1}\left(B\right)\backslash B$ with $\mu\left(A\right)>0$. Since $\mu\left(A\cap T^{-1}\left(A\right)\right)=0$ we have $\mu_{\tau}^{\ast}\left(N_{A}\circ T^{\ast}>0,N_{A}=0\right)>0$, hence $T^{\ast}\neq\mathrm{Id}_{X_{\tau}^{\ast}}\in\mathrm{Aut}\left(X_{\tau}^{\ast},\mathcal{B}_{\tau}^{\ast},\mu_{\tau}^{\ast}\right)$. The equivariance property \eqref{eq: equivariance} is verified by noting that for every $A\in\mathcal{B}_{\mu}$, for $\omega$ in an appropriate $\mu^{\ast}$-conull set,
$$N_{A}\circ T_{g}^{\ast}\left(\omega\right)=N_{A}\left(\omega\circ T_{g}^{-1}\right)=\omega\left(T_{g}^{-1}\left(A\right)\right)=N_{T_{g}^{-1}\left(A\right)}\left(\omega\right).$$
The continuity of this embedding can be verified using elementary considerations, but it is also an immediate consequence of the automatic continuity property of $\mathrm{Aut}\left(X,\mathcal{B},\mu\right)$ by Le Ma\^{\i}tre \cite[Theorem 1.2]{le2022polish}.
\end{proof}

\begin{dfn}[Classical Poisson point process]
The construction in Proposition \ref{Proposition: classical Poisson point process}, while depending on the highly non-canonical choice of $\tau$, serves as a concrete Poisson point process with an arbitrary base space $\left(X,\mathcal{B},\mu\right)$. Ignoring $\tau$, we will refer to it as the {\bf classical Poisson point process} and denote it by
$$\left(X^{\ast},\mathcal{B}^{\ast},\mu^{\ast}\right)\text{ and }\mathcal{N}=\left\{N_{A}:A\in\mathcal{B}\right\}.$$
\end{dfn}

While the choice of $\tau$ affects directly $X^{\ast}=X_{\tau}^{\ast}$ as a subspace of the $\tau$-Radon measures on $X$, the following proposition shows that the Poisson point process is a universal object to which the choice of $\tau$ is irrelevant up to a Borel isomorphism. In fact, we will show that the Poisson point process is universal in the widest sense of Definition \ref{Definition: Poisson point process} up to a Borel isomorphism.

\begin{prop}
\label{Proposition: Poisson point process}
All generative Poisson point processes on the same base space are isomorphic. More explicitly, let $\left(X,\mathcal{B},\mu\right)$ be a standard measure space and $\mathcal{P}=\left\{P_{A}:A\in\mathcal{B}\right\}$ be a generative Poisson point process that is defined on $\left(\Omega,\mathcal{F},\mathbb{P}\right)$ with base space $\left(X,\mathcal{B},\mu\right)$. There is an isomorphism of measure spaces
$$\varphi:\left(\Omega,\mathcal{F},\mathbb{P}\right)\to\left(X^{\ast},\mathcal{B}^{\ast},\mu^{\ast}\right),$$
such that on a $\mathbb{P}$-conull set,
$$N_{A}\circ\varphi=P_{A}\text{ for all }A\in\mathcal{B}.$$
Thus, $\mathcal{P}$ is a random measure on $X$ in that for every $\omega$ in an appropriate $\mathbb{P}$-conull set, the map $A\mapsto P_{A}\left(\omega\right)$ defines a measure on $\left(X,\mathcal{B}\right)$.
\end{prop}

\begin{proof}
For $\omega\in\Omega$ define $\varphi\left(\omega\right):\mathcal{B}\to\mathbb{R}_{\geq 0}$ by
$$\varphi\left(\omega\right)\left(A\right)=P_{A}\left(\omega\right).$$
First we prove that for $\omega$ in a $\mathbb{P}$-conull set it holds that $\varphi\left(\omega\right)\in X^{\ast}$, namely that $\varphi\left(\omega\right)$ extends to a genuine measure on $\mathcal{B}$. To this end we verify the conditions appears in \cite[p. 17]{daley2008introduction}. The finite additivity of $\varphi\left(\omega\right)$ for every $\omega\in\Omega$ is immediate from the definition of $\mathcal{P}$ as a Poisson point process. As for the continuity, we note that the finite additivity implies that $P_{A}\leq P_{B}$ whenever $A\subset B$, hence if $\mathcal{B}_{\mu}\ni A_n\searrow\emptyset$ as $n\to\infty$ then the pointwise limit of the monotone descending sequence $P_{A_{n}}$ as $n\to\infty$ is a nonnegative random variable and, by the dominated convergence theorem, it has zero mean, hence $P_{A_{n}}\searrow 0$ as $n\to\infty$, establishing the continuity. It follows from \cite[Lemma 9.1.XIV]{daley2008introduction} that $\varphi\left(\omega\right)\in X^{\ast}$ for $\omega$ on a $\mathbb{P}$-conull set. Thus we obtain a map $\varphi:\Omega\to X^{\ast}$ that is defined on an appropriate $\mathbb{P}$-conull set. In order to see that $\mathbb{P}\circ\varphi^{-1}=\mu^{\ast}$, note that for every $\omega$ in a $\mathbb{P}$-conull set,
$$N_{A}\left(\varphi\left(\omega\right)\right)=\varphi\left(\omega\right)\left(A\right)=P_{A}\left(\omega\right),\quad A\in\mathcal{B},$$
from which it readily follows that $\mathcal{N}$ forms a Poisson point process on the same base space with respect to $\mathbb{P}\circ\varphi^{-1}$ as well. The uniqueness of the classical Poisson point process implies that $\mathbb{P}\circ\varphi^{-1}=\mu^{\ast}$.
\end{proof}

\section{Foundations of Poissonian actions}

\subsection{Preliminaries: actions and representations of Polish groups}

The very definition of 'measure preserving action' of a Polish group has more than one possible meaning, essentially two meanings, which is the source of a substantial part of our study. The first and more general definition can be put in two ways, namely {\it near actions}, as was put by Zimmer, and {\it Boolean actions}, a classical object that admit a convenient formulation due to Glasner, Tsirelson \& Weiss (see \cite[Introduction]{glasnertsirelsonweiss2005} and the references therein). The other, more restrictive notion of spatial actions will be presented in the second part of this work, starting in Section \ref{Part: Spatial Poissonian actions}.

\begin{dfn}[Zimmer]
\label{Definition: near action}
A {\bf near action} of a Polish group $G$ on a standard measure space $\left(X,\mathcal{B},\mu\right)$ is a jointly measurable map $\mathbf{T}:G\times X\to X$, $\mathbf{T}:\left(g,x\right)\mapsto T_{g}\left(x\right)$, such that:
\begin{enumerate}
    \item $T_{e}=\mathrm{Id}_{X}$ on a $\mu$-conull set, where $e\in G$ is the identity element.
    \item $T_{g}\circ T_{h}=T_{gh}$ on a $\mu$-conull set for every $g,h\in G$.\footnote{It is crucial here that the $\mu$-conull set may depend on $g,h$.}
    \item $\mu\circ T_{g}^{-1}=\mu$ for every $g\in G$.
\end{enumerate}
\end{dfn}

There is a natural way to view $\mathrm{Aut}\left(X,\mathcal{B},\mu\right)$ as the group of Boolean isometries of the measure algebra associated with $\left(X,\mathcal{B},\mu\right)$ (i.e. the Boolean algebra of Borel sets in $X$ modulo $\mu$-null sets, with its natural complete metric). With this point of view, Glasner, Tsirelson \& Weiss put the notion of Boolean action as follows.

\begin{dfn}[Glasner, Tsirelson \& Weiss]
\label{Definition: Boolean action}
A {\bf Boolean action} of a Polish group $G$ on a standard measure space $\left(X,\mathcal{B},\mu\right)$ is a continuous (or equivalently, measurable)\footnote{The equivalence of measurability and continuity for homomorphisms between Polish group is by Pettis' automatic continuity (see e.g. \cite[$\mathsection$9.C]{kechris2012classical}).} homomorphism $\mathbf{T}:G\to\mathrm{Aut}\left(X,\mathcal{B},\mu\right)$.
\end{dfn}

The reader may recall Proposition \ref{Proposition: classical Poisson point process} that suggests why the formulation of Boolean action is the one that is more convenient for our purposes. However, as it was observed by Glasner, Tsirelson \& Weiss \cite[Introduction]{glasnertsirelsonweiss2005}, both definitions are essentially the same. Thus, we relate to near actions and Boolean actions simply as {\bf actions}, and denote this unified object by
$$\mathbf{T}:G\curvearrowright\left(X,\mathcal{B},\mu\right).$$
We may refer to an action as a {\it finite action} or an {\it infinite action}, to indicate whether the underlying measure is a probability measure or an infinite one.

A pair of actions $\mathbf{T}:G\curvearrowright\left(X,\mathcal{B},\mu\right)$ and $\mathbf{T}':G\curvearrowright\left(X',\mathcal{B}',\mu'\right)$ are considered to be isomorphic, if there exists a bi-measurable bijection $\varphi:X\to X'$, that is possibly defined only on corresponding conull sets, such that $\mu\circ\varphi^{-1}=\mu'$ and $T'_{g}\circ\varphi=\varphi\circ T_{g}$ for each $g\in G$ on a $\mu$-conull set.

Recall that a unitary representation $\mathbf{U}$ of a Polish group $G$ on a Hilbert space $\mathcal{H}$ is a group homomorphism $\mathbf{U}:G\to\mathrm{U}\left(\mathcal{H}\right)$, $\mathbf{U}:g\mapsto U_{g}$, where $\mathrm{U}\left(\mathcal{H}\right)$ denotes the unitary group of $\mathcal{H}$, such that the mapping $\left(g,f\right)\mapsto U_{g}f$ is jointly continuous.\footnote{In fact, from \cite[Theorem 4.8.6]{srivastava2008course} or \cite[Exercise (9.16) i)]{kechris2012classical} it follows that if the mapping $\left(g,f\right)\mapsto U_{g}f$ is jointly measurable then it is automatically jointly continuous.} When $\mathcal{H}$ is a subspace of some $L_{\mathbb{R}}^{2}$-space, we say that $\mathbf{U}$ is {\bf unital} if
$U_{g}1=1$ for every $g\in G$ (when $1$ is integrable), and that $\mathbf{U}$ is {\bf positivity preserving} if $f\geq 0$ implies $U_{g}f\geq 0$ for every $g\in G$ (see \cite[p. 207]{sato1999levy}). The (real) {\bf Koopman representation} of an action $\mathbf{T}:G\curvearrowright\left(X,\mathcal{B},\mu\right)$ is the unitary representation 
$$\mathbf{U}:G\to\mathrm{U}\left(L_{\mathbb{R}}^{2}\left(X,\mathcal{B},\mu\right)\right),\quad U_{g}:f\mapsto f\circ T_{g}^{-1},\quad g\in G.$$

\subsection{Poissonian actions}
\label{Part: Poissonian actions}

Recall Definition \ref{Definition: Poissonian action} for the general notion of Poissonian action. Let us now introduce the Poisson suspension construction in a precise way as a natural source for Poissonian actions. Suppose $\mathbf{T}:G\curvearrowright\left(X,\mathcal{B},\mu\right)$ is an action of a Polish group $G$ on a standard measure space $\left(X,\mathcal{B},\mu\right)$. Consider the classical Poisson point process $\mathcal{N}=\left\{N_{A}:A\in\mathcal{B}\right\}$ that is defined on $\left(X^{\ast},\mathcal{B}^{\ast},\mu^{\ast}\right)$, and using the continuous embedding introduced in Proposition \ref{Proposition: classical Poisson point process}, we obtain an action $\mathbf{T}^{\ast}:G\curvearrowright\left(X^{\ast},\mathcal{B}^{\ast},\mu^{\ast}\right)$ by composing
$$\mathbf{T}^{\ast}:G\to\mathrm{Aut}\left(X,\mathcal{B},\mu\right)\hookrightarrow\mathrm{Aut}\left(X^{\ast},\mathcal{B}^{\ast},\mu^{\ast}\right),\quad g\mapsto T_{g}\mapsto T_{g}^{\ast}.$$
The equivariance property \eqref{eq: equivariance} readily implies that the Koopman representation of $\mathbf{T}^{\ast}$ preserves $H\left(\mathcal{P}\right)$, thus $\mathbf{T}^{\ast}$ is a Poissonian action with respect to $\mathcal{N}$, and its base action is nothing but $\mathbf{T}:G\curvearrowright\left(X,\mathcal{B},\mu\right)$.

\begin{rem}
Note that in the setting of Definition \ref{Definition: base action}, if $\mathcal{P}$ is not generative we may move to the sub-σ-algebra on $\Omega$ that is generated by $\mathcal{P}$. Then if the equivariance relations of $\mathbf{S}$ and $\mathbf{T}$ hold, this sub-σ-algebra is $\mathbf{S}$-invariant and we obtain a factor of $\mathbf{S}$ which is a Poissonian action.
\end{rem}

We now formulate and prove each of the statements in Theorem \ref{Theorem: Poissonian actions}.

\begin{thm}
\label{Theorem: Poissonian actions I}
Let $\mathbf{T}:G\curvearrowright\left(X,\mathcal{B},\mu\right)$ be an action and $\mathcal{P}=\left\{P_{A}:A\in\mathcal{B}\right\}$ be a generative Poisson point process that is defined on $\left(\Omega,\mathcal{F},\mathbb{P}\right)$ with base space $\left(X,\mathcal{B},\mu\right)$. There exists a Poissonian action $\mathbf{S}:G\curvearrowright\left(\Omega,\mathcal{F},\mathbb{P}\right)$ with respect to $\mathcal{P}$ whose base action is $\mathbf{T}$. Moreover, $\mathbf{S}$ is essentially unique in that Poissonian actions associated with isomorphic actions are isomorphic.
\end{thm}

\begin{proof}
By Proposition \ref{Proposition: Poisson point process} we may assume without loss of generality that $\left(\Omega,\mathcal{F},\mathbb{P}\right)=\left(X^{\ast},\mathcal{B}^{\ast},\mu^{\ast}\right)$ and that $\mathcal{P}=\mathcal{N}$. Then the aforementioned construction of the Poisson suspension, using the embedding described in Proposition \ref{Proposition: classical Poisson point process}, is a Poissonian action with respect to $\mathcal{N}$ whose base action is $\mathbf{T}$.

In order to show the uniqueness, we start by showing that all Poissonian actions with base action $\mathbf{T}$ are isomorphic to the Poisson suspension $\mathbf{T}^{\ast}:G\curvearrowright\left(X^{\ast},\mathcal{B}^{\ast},\mu^{\ast}\right)$. Let $\mathbf{S}:G\curvearrowright\left(\Omega,\mathcal{F},\mathbb{P}\right)$ be such a Poissonian action with respect to a generative Poisson point process $\mathcal{P}=\left\{ P_{A}:A\in\mathcal{B}\right\}$ that is defined on $\left(\Omega,\mathcal{F},\mathbb{P}\right)$. By Proposition \ref{Proposition: Poisson point process} there is an isomorphism of probability spaces $\varphi:\Omega\to X^{\ast}$ such that $\mathbb{P}\circ\varphi^{-1}=\mu^{\ast}$, that interchanges $\mathcal{P}$ and $\mathcal{N}$ in that $N_{A}\circ\varphi=P_{A}$ for every $A\in\mathcal{B}$. In order to verify that indeed $\varphi\circ S_{g}=T_{g}^{\ast}\circ\varphi$ on a $\mathbb{P}$-conull set for every $g\in G$, note that
$$N_{A}\circ\varphi\circ S_{g}=P_{A}\circ S_{g}=P_{T_{g}^{-1}\left(A\right)}=N_{A}\circ T_{g}^{\ast}\circ\varphi\quad\text{ for every }A\in\mathcal{B},$$
and since $\mathcal{N}$ is generative the desired property follows. Thus, $\varphi$ is an isomorphism of actions.

Now for the general case, suppose that $\mathbf{T}':G\curvearrowright\left(X',\mu'\right)$ is an infinite action that is isomorphic to $\mathbf{T}:G\curvearrowright\left(X,\mathcal{B},\mu\right)$ through $\psi:X\to X'$. From the classical Poisson point process $\mathcal{N}'=\left\{ N_{A}':A\in\mathcal{B}'\right\} $ associated with $\left(X',\mu'\right)$ we obtain the Poisson point process $\mathcal{P}:=\left\{ N_{A}'\circ\psi^{-1}:A\in\mathcal{B}\right\} $ associated with $\left(X,\mathcal{B},\mu\right)$. It is then evident that the Poisson suspension $\mathbf{T}'^{\ast}:G\curvearrowright\left(X'^{\ast},\mu'^{\ast}\right)$ is a Poissonian action with base action $\mathbf{T}:G\curvearrowright\left(X,\mathcal{B},\mu\right)$ via the Poisson point process $\mathcal{P}$. Hence, by the previous part of the proof it is isomorphic to the Poisson suspension $\mathbf{T}^{\ast}:G\curvearrowright\left(X^{\ast},\mathcal{B}^{\ast},\mu^{\ast}\right)$.
\end{proof}

The following theorem generalizes the second statement of Theorem \ref{Theorem: Poissonian actions}, which can be seen as a statement on Koopman representations, to a larger family of unitary representations.

\begin{thm}
\label{Theorem: Poissonian actions II}
Suppose $\mathcal{P}$ is a generative Poisson point process that is defined on $\left(\Omega,\mathcal{F},\mathbb{P}\right)$ with base space $\left(X,\mathcal{B},\mu\right)$. Every unital positivity preserving unitary representation $\mathbf{U}:G\to\mathrm{U}\left(H\left(\mathcal{P}\right)\right)$, admits a Poissonian action $\mathbf{S}:G\curvearrowright\left(\Omega,\mathcal{F},\mathbb{P}\right)$ with respect to $\mathcal{P}$, whose Koopman representation is $\mathbf{U}$. Moreover, $\mathbf{S}$ admits an essentially unique base action $\mathbf{T}:G\curvearrowright\left(X,\mathcal{B},\mu\right)$.
\end{thm}

Noting that Koopman representations are unital and positivity preserving, we obtain from Theorem \ref{Theorem: Poissonian actions II} for Koopman representations:

\begin{cor}
\label{Corollary: Poissonian actions II}
Every Poissonian action arises from an essentially unique base action.
\end{cor}

For the proof of Theorem \ref{Theorem: Poissonian actions II} we introduce two lemmas. First, we will use the following substitution for the notion of unital Koopman operators when dealing with infinite measure spaces. A unitary operator $U$ of $L_{\mathbb{R}}^{2}\left(X,\mathcal{B},\mu\right)$ for some standard measure space $\left(X,\mathcal{B},\mu\right)$ will be called {\bf quasi-unital} if
$$\int_{X}Ufd\mu=\int_{X}fd\mu\text{ for every }f\in L_{\mathbb{R}}^{1}\left(X,\mathcal{B},\mu\right)\cap L_{\mathbb{R}}^{2}\left(X,\mathcal{B},\mu\right).$$
In particular, $U$ preserves $L_{\mathbb{R}}^{1}\left(X,\mathcal{B},\mu\right)\cap L_{\mathbb{R}}^{2}\left(X,\mathcal{B},\mu\right)$. The group of quasi-unital positivity preserving unitary operators of $L_{\mathbb{R}}^{2}\left(X,\mathcal{B},\mu\right)$ is a closed subgroup of the unitary group of $L_{\mathbb{R}}^{2}\left(X,\mathcal{B},\mu\right)$, hence it is a Polish group.

In the following we formulate in our terminology a well-known fact which is a version of the Banach-Lamperti Theorem for $L^{2}$. As it was observed in \cite[footnote 3]{tulcea1964ergodic}, while the general Banach-Lamperti Theorem is formulated for unitary operators of $L^{p}$-spaces for $p\neq 2$, for positivity preserving unitary operators the proof of Lamperti applies for $L^{2}$-spaces as well. A byproduct of the following lemma is that when $\mu$ is a probability measure, for a positivity preserving unitary operator of $L_{\mathbb{R}}^{2}\left(X,\mathcal{B},\mu\right)$ being unital and quasi-unital is the same.

\begin{lem}
\label{Lemma: Banach-Lamperti}
For every standard measure space $\left(X,\mathcal{B},\mu\right)$, the Koopman embedding
$$T\mapsto U_{T},\quad U_{T}f=f\circ T^{-1},$$
forms an isomorphism of Polish groups between $\mathrm{Aut}\left(X,\mathcal{B},\mu\right)$ and the group of quasi-unital positivity preserving unitary operators of $L_{\mathbb{R}}^{2}\left(X,\mathcal{B},\mu\right)$.
\end{lem}

\begin{proof}
By the Banach-Lamperti Theorem in $L^{2}$ (see the aforementioned \cite[footnote 3]{tulcea1964ergodic}), there is a bijective correspondence between the group of nonsingular transformations of $\left(X,\mathcal{B},\mu\right)$ and the group of positivity preserving unitary operators of $L_{\mathbb{R}}^{2}\left(X,\mathcal{B},\mu\right)$, that is given by
$$T\mapsto U_{T},\quad U_{T}f=\sqrt{\frac{d\mu\circ T^{-1}}{d\mu}}f\circ T^{-1}.$$
This is a homomorphism of groups and, since $T$ is necessarily measure preserving when $U_{T}$ is quasi-unital, the measure preserving transformations correspond to the quasi-unital positivity preserving unitary operators. Thus, the restriction of this homomorphism to $\mathrm{Aut}\left(X,\mathcal{B},\mu\right)$, which is the usual Koopman embedding, forms a bijective homomorphism from $\mathrm{Aut}\left(X,\mathcal{B},\mu\right)$ onto the closed subgroup of quasi-unital positivity preserving unitary operators of $L_{\mathbb{R}}^{2}\left(X,\mathcal{B},\mu\right)$. Finally, the Polish topology of $\mathrm{Aut}\left(X,\mathcal{B},\mu\right)$ is, by definition, induced from this correspondance, so this is a homeomorphism.
\end{proof}

For the next step toward proving Theorem \ref{Theorem: Poissonian actions II} we will take a further look into unitary operators of $H\left(\mathcal{P}\right)$. The following objects and their basic properties are presented in more details in Appendix \ref{Appendix: The First Chaos of a Poisson Point process}. Fix a Poisson point process $\mathcal{P}$ as in Definition \ref{Definition: Poisson point process}. The {\bf first chaos} of $\mathcal{P}$ is the space
$$H_{1}\left(\mathcal{P}\right):=\overline{\mathrm{span}}\left\{ P_{A}-\mu\left(A\right):A\in\mathcal{B}_{\mu}\right\} \subset L_{\mathbb{R}}^{2}\left(\Omega,\mathcal{F},\mathbb{P}\right).$$
As part of the Fock space structure associated with $\mathcal{P}$, there is an isometric isomorphism of Hilbert spaces
$$I_{\mu}:L_{\mathbb{R}}^{2}\left(X,\mathcal{B},\mu\right)\to H_{1}\left(\mathcal{P}\right),\quad I_{\mu}:f\mapsto I_{\mu}\left(f\right),$$
that is given by a stochastic integral against $\mathcal{P}$ in an appropriate sense. Recalling the space
$$H\left(\mathcal{P}\right)=\overline{\mathrm{span}}\left\{ P_{A}:A\in\mathcal{B}_{\mu}\right\} \subset L_{\mathbb{R}}^{2}\left(\Omega,\mathcal{F},\mathbb{P}\right),$$
the first chaos is its direct summand,
$$H\left(\mathcal{P}\right)=H_{1}\left(\mathcal{P}\right)\oplus\mathbb{R}.$$
For every unital positivity preserving unitary operator $U$ of $H\left(\mathcal{P}\right)$ we have
$$\left\langle U\left(I_{\mu}\left(f\right)\right),1\right\rangle =\left\langle I_{\mu}\left(f\right),U^{-1}\left(1\right)\right\rangle =\left\langle I_{\mu}\left(f\right),1\right\rangle ,\quad f\in L_{\mathbb{R}}^{2}\left(X,\mathcal{B},\mu\right).$$
Thus, $U$ preserves $H_{1}\left(\mathcal{P}\right)$ as a direct summand of $H\left(\mathcal{P}\right)$. Since $I_{\mu}$ is an isometric isomorphism of Hilbert spaces, the conjugation by $I_{\mu}$ forms a map between the unitary groups,
\begin{equation}
\label{eq: stochastic operator}
\mathrm{U}\left(H\left(\mathcal{P}\right)\right)\to\mathrm{U}\left(L_{\mathbb{R}}^{2}\left(X,\mathcal{B},\mu\right)\right),\quad U\mapsto U^{\mu}:=I_{\mu}^{-1}\circ U\circ I_{\mu}.
\end{equation}
This map will be important to the proof of Theorem \ref{Theorem: Poissonian actions II}. The following lemma deals with its fundamental properties, with its proof following a similar approach to the proof of \cite[Proposition 4.4]{roy2009poisson}.

\begin{lem}
\label{Lemma: positivity preserving unitaries}
In the map $U\mapsto U^{\mu}$ as in \eqref{eq: stochastic operator}, unital positivity preserving unitary operators of $H\left(\mathcal{P}\right)$ are mapped to quasi-unital positivity preserving unitary operator of $L_{\mathbb{R}}^{2}\left(X,\mathcal{B},\mu\right)$.
\end{lem}

The proof makes essential use of the properties of the Poisson stochastic integral $I_{\mu}$ as well as the fundamentals of infinitely divisible Poissonian (henceforth {\sc id}p) random variables. We present this in Appendix \ref{Appendix: The First Chaos of a Poisson Point process}.

\begin{proof}
Since $I_{\mu}$ is an isometric isomorphism of Hilbert spaces, $U^{\mu}$ is a unitary operator. Let us fix an arbitrary $f\in L_{\mathbb{R}}^{1}\left(X,\mathcal{B},\mu\right)\cap L_{\mathbb{R}}^{2}\left(X,\mathcal{B},\mu\right)$. As described in Proposition \ref{Proposition: First chaos}, we have the {\sc id}p random variable
$$W_{f}:=I_{\mu}\left(U^{\mu}f\right)+\int_{X}fd\mu=U\left(I_{\mu}\left(f\right)\right)+\int_{X}fd\mu=U\Big(\int_{X}fd\mathcal{P}\Big),$$
whose L\'{e}vy measure is given by
$$\ell_{U^{\mu}f}=\mu\mid_{\left\{U^{\mu}f\neq 0\right\}}\circ \left(U^{\mu}f\right)^{-1}.$$
Assume further that $f\geq 0$ so that also $\int_{X}fd\mathcal{P}\geq 0$ and, since $U$ is positivity preserving, also $W_{f}\geq 0$. From Proposition \ref{Proposition: First chaos}(3) it follows that
$$\mu\left(U^{\mu}f<0\right)=\ell_{U^{\mu}f}\left(\mathbb{R}_{<0}\right)=0\text{ and }\int_{X}U^{\mu}fd\mu=\int_{\mathbb{R}_{\geq 0}}td\ell_{U^{\mu}f}\left(t\right)<\infty.$$
The first property shows that $U^{\mu}$ preserves positivity for every $f$ in a dense subspace, hence it is positivity preserving. The second property shows that $U^{\mu}f\in L_{\mathbb{R}}^{1}\left(X,\mathcal{B},\mu\right)\cap L_{\mathbb{R}}^{2}\left(X,\mathcal{B},\mu\right)$, so that by Proposition \ref{Proposition: First chaos}(1),
$$I_{\mu}\left(U^{\mu}f\right)=\int_{X}U^{\mu}fd\mathcal{P}-\int_{X}U^{\mu}fd\mu.$$
Plugging this into the definition of $W_{f}$ and using that $W_{f}\geq 0$, we obtain
$$\int_{X}U^{\mu}fd\mu-\int_{X}fd\mu=\int_{X}U^{\mu}fd\mathcal{P}-W_{f}\leq\int_{X}U^{\mu}fd\mathcal{P}.$$
With $f$ being fixed, the left hand side is a constant, while the right hand side is a nonnegative {\sc id}p random variable that is obtained as a stochastic integral of an integrable function. It follows from \cite[Corollary 24.8]{sato1999levy} that the infimum of the right hand side (as a random variable) is zero, and we conclude that
$$\int_{X}U^{\mu}fd\mu\leq\int_{X}fd\mu.$$
Since the map \eqref{eq: stochastic operator} respects inverses, the same proof shows that the same inequality holds when $U^{\mu}$ is replaced by $\left(U^{\mu}\right)^{-1}$. Both inequalities readily imply that
$$\int_{X}U^{\mu}fd\mu=\int_{X}fd\mu,$$
hence $U^{\mu}$ is quasi-unital.
\end{proof}

\begin{proof}[Proof of Theorem \ref{Theorem: Poissonian actions II}]
Suppose $\mathbf{U}:G\to\mathrm{U}\left(H\left(\mathcal{P}\right)\right)$ is a unital positivity preserving unitary representation of a Polish group $G$ as in the proposition. We construct a continuous homomorphism along the following arrows
$$G\xrightarrow{\mathbf{U}}\mathrm{U}\left(H\left(\mathcal{P}\right)\right)\xrightarrow{\eqref{eq: stochastic operator}}\mathrm{U}\left(L_{\mathbb{R}}^{2}\left(X,\mathcal{B},\mu\right)\right)\xrightarrow{\text{Banach-Lamperti}}\mathrm{Aut}\left(X,\mathcal{B},\mu\right)$$
as follows. The first arrow is the given representation $g\mapsto U_{g}$. The second arrow is the map $U_{g}\mapsto U_{g}^{\mu}$ as in \eqref{eq: stochastic operator}, whose image lies in the closed subgroup of quasi-unital positivity preserving unitary operators by Lemma \ref{Lemma: positivity preserving unitaries}. Then on the image of the second arrow, the third arrow $U_{g}^{\mu}\mapsto T_{g}$ is given by Lemma \ref{Lemma: Banach-Lamperti}. Since all the arrows are continuous, the map $g\mapsto T_{g}$ constitutes an action $\mathbf{T}:G\curvearrowright\left(X,\mathcal{B},\mu\right)$. We now use Theorem \ref{Theorem: Poissonian actions I}, and take a Poissonian action $\mathbf{S}:G\curvearrowright\left(\Omega,\mathcal{F},\mathbb{P}\right)$ whose base action is $\mathbf{T}$. The equivariance property that relates $\mathbf{T}$ and $\mathbf{S}$ as in Definition \ref{Definition: base action} implies that
$$U_{g}\left(I_{\mu}\left(f\right)\right)=I_{\mu}\left(U_{g}^{\mu}f\right)=I_{\mu}\left(f\circ T_{g}^{-1}\right)=I_{\mu}\left(f\right)\circ S_{g},\quad g\in G,$$
for every $f\in L_{\mathbb{R}}^{1}\left(X,\mathcal{B},\mu\right)\cap L_{\mathbb{R}}^{2}\left(X,\mathcal{B},\mu\right)$. Thus, $\mathbf{U}$ is the Koopman representation of $\mathbf{S}$, which is a Poissonian action with base action $\mathbf{T}$ as required in the theorem.

It is clear that all actions whose Koopman representation is $\mathbf{U}$ are isomorphic, hence $\mathbf{S}$ is essentially unique. In order to see the essential uniqueness of $\mathbf{T}$, we note that if $\mathbf{T}'$ is another base action for $\mathbf{S}$ then
$$P_{T_{g}^{-1}\left(A\right)}=P_{A}\circ S_{g}=P_{T_{g}'^{-1}\left(A\right)}\quad\mathbb{P}\text{-a.e. for every }g\in G\text{ and }A\in\mathcal{B}.$$
However, just as in the uniqueness argument in the proof of Proposition \ref{Proposition: classical Poisson point process}, if for some $g\in G$ we have $T_{g}\neq T_{g}'$, then for $A\in\mathcal{B}_{\mu}$ for which $\mu\left(T_{g}^{-1}\left(A\right)\cap T_{g}'^{-1}\left(A\right)\right)=0$ it holds that
$$\mathbb{P}\big(P_{T_{g}^{-1}\left(A\right)}>0,P_{T_{g}'^{-1}\left(A\right)}=0\big)>0,$$
which is a contradiction. This completes the proof.
\end{proof}

\subsection{Ergodicity of Poissonian actions}

Here we prove Theorem \ref{Theorem: ergodicity}. Let us first recall some basic definitions. For an action $\mathbf{T}:G\curvearrowright\left(X,\mathcal{B},\mu\right)$ we denote the invariant σ-algebra
$$\mathcal{I}_{\mathbf{T}}:=\left\{A\in\mathcal{B}:\mu\left(A\triangle T_g^{-1}\left(A\right)\right)=0\text{ for every }g\in G\right\}.$$
We say that $\mathbf{T}$ is:
\begin{description}
    \item[$\bullet$ Null] for every $A\in\mathcal{I}_{\mathbf{T}}$ either $\mu\left(A\right)=0$ or $\mu\left(A\right)=\infty$;
    \item[$\bullet$ Ergodic] for every  $A\in\mathcal{I}_{\mathbf{T}}$ either $\mu\left(A\right)=0$ or $\mu\left(X\backslash A\right)=0$;
    \item[$\bullet$ Doubly Ergodic] the diagonal action $\mathbf{T}\otimes\mathbf{T}:G\curvearrowright\left(X\times X,\mathcal{B}\otimes\mathcal{B},\mu\otimes\mu\right)$ is ergodic;
    \item[$\bullet$ Weakly mixing] for every ergodic action $\mathbf{S}:G\curvearrowright\left(Y,\mathcal{C},\nu\right)$ of $G$, the diagonal action $\mathbf{T}\otimes\mathbf{S}:G\curvearrowright\left(X\times Y,\mathcal{B}\otimes\mathcal{C},\mu\otimes\nu\right)$ is ergodic. 
\end{description}

\begin{rem}
\label{Remark: weak mixing}
A few general remarks about those properties:
\begin{enumerate}
    \item Being null is equivalent to the non-existence of a $\mathbf{T}$-invariant probability measure that is absolutely continuous with respect to $\mu$.
    \item Double ergodicity and weak mixing are equivalent in probability preserving actions of general groups. For locally compact groups this is a classical fact, and it was pointed out to us by Benjy Weiss in a personal communication that this is true in full generality. Indeed, one implication is obvious, and the other was proved by Glasner \& Weiss in \cite[Theorem 2.1]{glasner2016weak}, that while is formulated for locally compact groups the proof holds in full generality. With the terminology of \cite[Definition 1.1]{glasner2016weak}, this can be seen by looking at the proofs of the implications $DE\implies EIC\implies EUC\implies WM$.
    \item As a consequence of the L\'{e}vy 0-1 Law, weak mixing is equivalent to that the infinite diagonal action $\mathbf{T}^{\otimes\mathbb{N}}:G\curvearrowright\left(X^{\mathbb{N}},\mathcal{B}^{\otimes\mathbb{N}},\mu^{\otimes\mathbb{N}}\right)$ is ergodic. This will be useful in the proof of Theorem \ref{Theorem: Diff spatial action}.
\end{enumerate}
\end{rem}

In proving Theorem \ref{Theorem: ergodicity} we will need the following key result that was established in \cite[Theorem 1]{parreau2015prime}. We will refer to the chaos structure of $L_{\mathbb{R}}^{2}\left(X^{\ast},\mathcal{B}^{\ast},\mu^{\ast}\right)$ that arises from its structure as the Fock space of $L_{\mathbb{R}}^{2}\left(X,\mathcal{B},\mu\right)$ as we introduce in Appendix \ref{Appendix: Chaos decomposition}.

\begin{thm}[Parreau \& Roy]
\label{Theorem: Parreau & Roy}
Let $\left(X,\mathcal{B},\mu\right)$ be a standard infinite measure space. Every orthogonal projection of $L^2\left(X^{\ast},\mathcal{B}^{\ast},\mu^{\ast}\right)$ that preserves its chaos structure and vanishes on the first chaos $H_{1}\left(\mathcal{N}\right)$, necessarily vanishes on all higher chaoses, i.e. it is the projection to the constants.
\end{thm}

The following technical lemma was proved for a single transformation in \cite[$\mathsection$ 3.2]{roy2021}, and the following extension to general groups is straightforward.

\begin{lem}
\label{Lemma: chaos}
Let $\mathbf{T}:G\curvearrowright\left(X,\mathcal{B},\mu\right)$ be an action of a Polish group $G$, and consider the projection
$$\pi_{\mathbf{T}}:L_{\mathbb{R}}^{2}\left(X,\mathcal{B},\mu\right)\to L_{\mathbb{R}}^{2}\left(X,\mathcal{I}_{\mathbf{T}},\mu\right).$$
Then the projection
$$\pi_{\mathbf{T}^{\ast}}:L_{\mathbb{R}}^{2}\left(X^{\ast},\mathcal{B}^{\ast},\mu^{\ast}\right)\to L_{\mathbb{R}}^{2}\left(X^{\ast},\mathcal{I}_{\mathbf{T}^{\ast}},\mu^{\ast}\right).$$
preserves the chaos structure of $L_{\mathbb{R}}^{2}\left(X^{\ast},\mathcal{B}^{\ast},\mu^{\ast}\right)$.
\end{lem}

\begin{proof}
We start with a single transformation $T\in\mathrm{Aut}\left(X^{\ast},\mathcal{B}^{\ast},\mu^{\ast}\right)$. Let $U_{T}$ be the Koopman operator of $T$ and set $\Psi_{T}=U_{T}-\mathrm{Id}$, noting that $\mathrm{Im}\pi_{T}=\ker\Psi_{T}$. Since $U_{T}$ preserves the chaos structure of $L_{\mathbb{R}}^{2}\left(X^{\ast},\mathcal{B}^{\ast},\mu^{\ast}\right)$ then so is $\Psi_{T}$, so that for every $f=\sum_{n=0}^{\infty}f_{n}\in L_{\mathbb{R}}^{2}\left(X^{\ast},\mathcal{B}^{\ast},\mu^{\ast}\right)$, where $f_n\in H_{n}\left(X,\mathcal{B},\mu\right)$ for every $n$, we have
$$\Psi_{T}={\textstyle \sum_{n=0}^{\infty}}\Psi_{T}f_{n}\text{ and }\Psi_{T}f_{n}\in H_{n}\left(X,\mathcal{B},\mu\right)\text{ for every }n.$$
It follows that $f\in\ker\Psi_{T}$ if and only if $\Psi_{T}f_{n}=0$ for every $n$, that is to say
$$\ker\Psi_{T}={\textstyle \bigoplus_{n=0}^{\infty}}\left(\ker\Psi_{T}\cap H_{n}\left(X,\mathcal{B},\mu\right)\right).$$
However, we have that $\mathrm{Im}\pi_{T}=\ker\Psi_{T}$ on each $H_{n}\left(X,\mathcal{B},\mu\right)$, so we obtain
\begin{equation}
    \label{eq: chaoe image}
    \mathrm{Im}\pi_{T}={\textstyle \bigoplus_{n=0}^{\infty}}\left(\mathrm{Im}\pi_{T}\cap H_{n}\left(X,\mathcal{B},\mu\right)\right).
\end{equation}
Since $\ker\pi_{T}$ is the orthogonal complement of $\mathrm{Im}\pi_{T}$ on each $H_{n}\left(X,\mathcal{B},\mu\right)$, we deduce that also
\begin{equation}
    \label{eq: chaoe kernel}
    \ker\pi_{T}={\textstyle \bigoplus_{n=0}^{\infty}}\left(\ker\pi_{T}\cap H_{n}\left(X,\mathcal{B},\mu\right)\right).
\end{equation}
Then \eqref{eq: chaoe image} and \eqref{eq: chaoe kernel} imply that $\pi_{T}$ preserves the chaos structure of $L_{\mathbb{R}}^{2}\left(X^{\ast},\mathcal{B}^{\ast},\mu^{\ast}\right)$, proving the lemma for a single transformation.

Now for an action $\mathbf{T}$, if we let $\Psi_{g}=\Psi_{T_{g}^{\ast}}=U_{T_{g}^{\ast}}-\mathrm{Id}$ for each $g\in G$, then by the same reasoning
$${\textstyle \bigcap_{g\in G}}\ker\Psi_{g}={\textstyle \bigoplus_{n=0}^{\infty}\bigcap_{g\in G}}\left(\ker\Psi_{g}\cap H_{n}\left(X,\mathcal{B},\mu\right)\right).$$
Since $\mathrm{Im}\pi_{\mathbf{T}}=\bigcap_{g\in G}\ker\Psi_{g}$ on each $H_{n}\left(X,\mathcal{B},\mu\right)$, it similarly follows that
$$\mathrm{Im}\pi_{\mathbf{T}}={\textstyle \bigoplus_{n=0}^{\infty}}\left(\mathrm{Im}\pi_{\mathbf{T}}\cap H_{n}\left(X,\mathcal{B},\mu\right)\right)$$
and then that
$$\ker\pi_{\mathbf{T}}={\textstyle \bigoplus_{n=0}^{\infty}}\left(\ker\pi_{\mathbf{T}}\cap H_{n}\left(X,\mathcal{B},\mu\right)\right).$$
This completes the proof of the lemma also for actions.
\end{proof}

\begin{proof}[Proof of Theorem \ref{Theorem: ergodicity}]
By Proposition \ref{Proposition: Poisson point process} and Theorem \ref{Theorem: Poissonian actions} it is sufficient to prove the theorem for Poisson suspensions $\mathbf{T}^{\ast}:G\curvearrowright\left(X^{\ast},\mathcal{B}^{\ast},\mu^{\ast}\right)$. If the action is null then the projection $\pi_{\mathbf{T}}$ vanishes on the first chaos $H_{1}\left(\mathcal{N}\right)$. From Lemma \ref{Lemma: chaos} and Theorem \ref{Theorem: Parreau & Roy} we obtain that $\text{Im}\pi_{\mathbf{T}}$ consists of constant functions, namely the Poisson suspension is ergodic. Conversely, if the action is not null so that $\mathcal{I}_{\mathbf{T}}$ contains a set $A$ with $0<\mu\left(A\right)<\infty$, then $N_{A}$ is a non-constant, $\mathbf{T}^{\ast}$-invariant function in $L^2\left(X^{\ast},\mathcal{B}^{\ast},\mu^{\ast}\right)$, so the Poisson suspension is not ergodic.

Let us now show that for the Poisson suspension ergodicity implies weak mixing, and in doing so we will use the previous part of the proof twice. The Poisson suspension being ergodic implies that $\mathbf{T}:G\curvearrowright\left(X,\mathcal{B},\mu\right)$ is null, and it is clear that in this case also the diagonal action
$$\mathbf{T}\otimes\mathrm{Id}:G\curvearrowright\left(X\times\left\{0,1\right\},\mu\otimes{\textstyle\frac{1}{2}}\left(\delta_{0}+\delta_{1}\right)\right)$$
is null, so that the Poisson suspension associated with this latter action is ergodic as well. However, this Poisson suspension, $\left(\mathbf{T}\otimes\mathrm{Id}\right)^{\ast}$, when is taken with respect to the product of the topology $\tau$ for which $X^{\ast}$ was defined with the discrete topology of $\left\{0,1\right\}$, is isomorphic to $\mathbf{T}^{\ast}\otimes\mathbf{T}^{\ast}:G\curvearrowright\left(X^{\ast}\times X^{\ast},\mathcal{B}^{\ast},\mathcal{B}^{\ast},\mu^{\ast}\otimes\mu^{\ast}\right)$, hence $\mathbf{T}^{\ast}$ is weakly mixing.
\end{proof}

\begin{rem}
\label{Remark: ergodicity}
It was suggested to us by Sasha Danilenko that, since the ergodicity of a probability preserving action of a Polish group is determined by the ergodicity of the action restricted to a countable dense subgroup, it is sufficient to prove Theorem \ref{Theorem: ergodicity} for countable (non-topological) groups. Now for countable groups, Theorem \ref{Theorem: ergodicity} can be deduced from known results mentioned in \cite[Facts 2.2 \& 2.4]{danilenko2023krieger}.
\end{rem}

\section{Spatial Poissonian actions}
\label{Part: Spatial Poissonian actions}

Our main object of study in this part is a more restrictive notion of measure preserving actions, namely spatial actions.

\subsection{Preliminaries: spatial actions and the point-realization problem}
\label{Section: Preliminaries: spatial actions and the point-realization problem}

\begin{dfn}
\label{Definition: spatial action}
A (measure preserving) {\bf spatial action} of a Polish group $G$ on a standard measure space $\left(X,\mathcal{B},\mu\right)$ is a Borel action $\mathbf{T}:G\curvearrowright\left(X,\mathcal{B}\right)$, $\mathbf{T}:\left(g,x\right)\mapsto T_{g}\left(x\right)$, such that $\mu\circ T_{g}^{-1}=\mu$ for every $g\in G$. We denote spatial actions by
$$\mathbf{T}:G\spa\left(X,\mathcal{B},\mu\right).$$
\end{dfn}

Every spatial action induces a (near/Boolean) action in an obvious way, and in this case it can be thought of as a {\it point-realization} of the resulting (near/Boolean) action. However, in general, not every action admits a point-realization, and this leads to the point-realization problem in ergodic theory, which revolves around whether a given action admits a point-realization. As it turns out, for important classes of Polish groups this problem has a striking solution.

A Polish group $G$ is said to possess the {\bf Mackey property} (following \cite{kechris2012dynamics}) if every finite action of $G$ admits a point-realization. The following classes of groups are known to possess the Mackey property:
\begin{description}
    \item[$\bullet$ Locally compact Polish groups] This is the celebrated Mackey-Ramsay Point-Realization Theorem \cite[Theorem 3.3]{ramsay1971}.
    \item[$\bullet$ Non-Archimedean groups] Polish groups with a base of clopen subgroups at the identity. This is Glasner \& Weiss' \cite[Theorem 4.3]{glasner2005spatial}.
    \item[$\bullet$ Groups of isometries of a locally compact metric space] closed subgroups of the group of isometries of a locally compact metric space, with the Polish topology of pointwise convergence. This is Kwiatkowska \& Solecki's \cite[Theorem 1.1]{kwiatkowska2011spatial}.
\end{description}

The class of groups of isometries of a locally compact metric space includes both, locally compact Polish groups and non-Archimedean groups, and to the best of our knowledge this is the largest class of Polish groups on which the Mackey property is known to hold (beyond Polish groups see \cite{danilenko2000point}).

The fact that there are Polish groups without the Mackey property was demonstrated by Becker for the Abelian group $L^{0}\left(\left[0,1\right],S^{1}\right)$ of measurable functions $\left[0,1\right]\to S^{1}$, identified up to equality on a Lebesgue-conull set, with the topology of convergence in measure (see \cite[Appendix A]{glasnertsirelsonweiss2005}, \cite[$\mathsection$7.1]{pestov2006} and the references therein). 

This was vastly generalized by Glasner, Tsirelson \& Weiss \cite[Theorem 1.1]{glasnertsirelsonweiss2005} in showing that the the Mackey property fails miserably for the class of {\it L\'{e}vy groups}, a class of groups that was studied by many following Gromov \& Milman (see \cite[$\mathsection$ 4]{pestov2006} and the references therein). This class includes, among others, the group $\mathrm{Aut}\left(X,\mathcal{B},\mu\right)$ itself with the weak topology; the unitary group $\mathrm{U}\left(\mathcal{H}\right)$ of a separable Hilbert space $\mathcal{H}$ with the strong operator topology; and, the aforementioned $L^{0}\left(\left[0,1\right],S^{1}\right)$. Thus, it was shown by Glasner, Tsirelson \& Weiss that L\'{e}vy groups admit no finite spatial actions whatsoever except for trivial ones, and a fortiori do not possess the Mackey property. There are also non-L\'{e}vy Polish groups that do not possess the Mackey property \cite[$\mathsection$6]{glasner2005spatial}, \cite{moore2012boolean}.

\begin{rem}
\label{Remark: trivial}
A spatial action is considered to be trivial if the set of fixed points of the action has full measure. Note that the set of fixed points of a Borel action of a Polish group is a Borel set. This is clearly true for Polish actions, and for Borel action this follows from the theorem of Becker \& Kechris \cite[$\mathsection$ 5.2]{beckerkechris1996} by which every Borel action is a Polish action with respect to some Polish topology that induces the given Borel σ-algebra.
\end{rem}

Every Poissonian action of a group that possesses the Mackey property, admits a point-realization that serves as a spatial Poissonian action. By a {\it spatial Poissonian action} we refer to a Poissonian action (as in Definition \ref{Definition: Poissonian action}) which is also a spatial action. Our goal here is to show that the point-realization problem in Poissonian actions can be solved without appealing to the Mackey property.

\subsection{Poisson random set}
\label{Section: Poisson random set}

For a standard Borel space $\left(X,\mathcal{B}\right)$, with every Polish topology $\tau$ that generates $\mathcal{B}$ is associated the {\bf Effros Borel space}
$$\mathbf{F}_{\tau}\left(X\right)=\left\{F\subset X:X\backslash F\in\tau\right\}.$$
This is a standard Borel space with the \textbf{Effros σ-algebra} $\mathcal{E}_{\tau}\left(X\right)$, that is generated by the sets
$$B_{O}:=\left\{ F\in \mathbf{F}_{\tau}\left(X\right):F\cap O\neq\emptyset\right\},\quad O\in\tau.$$
See e.g. \cite[Section (12.C)]{kechris2012classical}. A {\bf random closed set} is nothing but a probability measure on $\left(\mathbf{F}_{\tau}\left(X\right),\mathcal{E}_{\tau}\left(X\right)\right)$, namely an element of $\mathcal{M}_{1}\left(\mathbf{F}_{\tau}\left(X\right),\mathcal{E}_{\tau}\left(X\right)\right)$. A common way to construct random closed sets for locally compact topologies is the Choquet Capacity Theorem (see e.g. \cite[$\mathsection$1.1.3]{molchanov2017}), but here we shall construct the Poisson random closed set also for Polish topologies that may not be locally compact.

Let us denote by
$$\left(\mathbf{F}_{\tau}^{\ast}\left(X\right),\mathcal{E}_{\tau}^{\ast}\left(X\right)\right)\subset\left(\mathbf{F}_{\tau}\left(X\right),\mathcal{E}_{\tau}\left(X\right)\right)$$
the subspace of infinite sets in $\mathbf{F}_{\tau}\left(X\right)$ with its induced Effros σ-algebra. Note that $\mathbf{F}_{\tau}^{\ast}\left(X\right)$ is indeed an Effros-measurable set: fixing a countable base $\mathcal{O}$ for $\tau$, for every $n\in\mathbb{Z}_{\geq 0}$ the property $\#F\geq n$ of $F\in\mathbf{F}_{\tau}\left(X\right)$ is witnessed by the existence of pairwise disjoint $O_{1},\dotsc,O_{n}\in\mathcal{O}$ such that $F\in B_{O_{1}}\cap\dotsm\cap B_{O_{n}}$.

\begin{thm}[Poisson random set]
\label{Theorem: Poisson random set, revisited}
Let $\left(X,\mathcal{B}\right)$ be a standard Borel space. For every Polish topology $\tau$ that generates $\mathcal{B}$, there are random variables 
$$\left\{\Xi_{A}:A\in\mathcal{B}\right\}\text{ of the form }\Xi_{A}:\mathbf{F}_{\tau}^{\ast}\left(X\right)\to\mathbb{Z}_{\geq 0}\cup\left\{ \infty\right\},$$
and a one-to-one correspondence 
$$\mathcal{M}_{\sigma}^{\tau}\left(X,\mathcal{B}\right)\to\mathcal{M}_{1}\left(\mathbf{F}_{\tau}^{\ast}\left(X\right),\mathcal{E}_{\tau}^{\ast}\left(X\right)\right),\quad \mu\mapsto\mu_{\tau}^{\ast},$$
such that $\left\{\Xi_A:A\in\mathcal{B}\right\}$ forms a Poisson point process with base space $\left(X,\mathcal{B},\mu\right)$. Furthermore, for every $\mu$ the following properties hold.
\begin{enumerate}
    \item $\mu_{\tau}^{\ast}$ is supported on the class of $\tau$-discrete sets.\footnote{While the class of $\tau$-discrete sets is generally not Effros-measurable, by a theorem of Hurewicz it is co-analytic (see \cite[Theorem (27.5), Exercise (27.8)]{kechris2012classical}) hence universally measurable.}
    \item On the support of $\mu_{\tau}^{\ast}$ it holds that $\Xi_{A}\left(\cdot\right)=\#\left(A\cap\cdot\right)$ for every $A\in\mathcal{B}$.
\end{enumerate}
\end{thm}

In order to spot the inherent difficulty in defining the Poisson point process as a random closed set in $\mathbf{F}_{\tau}\left(X\right)$, it should be noted that as long as $\tau$ is not locally compact, then for a general $A\in\mathcal{B}$, the function
$$\mathbf{F}_{\tau}\left(X\right)\to\mathbb{Z}_{\geq 0}\cup\left\{\infty\right\},\quad F\mapsto\#\left(A\cap F\right),$$
may not be Effros-measurable unless $A\in\tau$. In fact, by a theorem of Christensen this function may fail to be Effros-measurable even if $A\in\mathbf{F}_{\tau}\left(X\right)$ \cite[Theorem (27.6)]{kechris2012classical}. In order to resolve this issue we use the following simple modification on the Kuratowski \& Ryll-Nardzewski's Selection Theorem.

\begin{prop}
\label{Proposition: measurable injective selection}
Let $\left(X,\mathcal{B}\right)$ be a standard Borel space. For every Polish topology $\tau$ that generates $\mathcal{B}$  there is a countable collection of mappings
$$\left\{\xi_{n}:n\in\mathbb{N}\right\}\text{ of the form }\xi_{n}:\mathbf{F}_{\tau}^{\ast}\left(X\right)\to X,$$
such that the following properties hold.
\begin{enumerate}
\item (Measurability) Each $\xi_{n}$ is Effros-measurable.
\item (Injectivity) For each $F\in \mathbf{F}_{\tau}^{\ast}\left(X\right)$, the mapping $\mathbb{N}\to X$ given by $n\mapsto \xi_{n}\left(F\right)$ is injective.
\item (Selectivity) For each $F\in \mathbf{F}_{\tau}^{\ast}\left(X\right)$, the set $\left\{\xi_{n}\left(F\right):n\in\mathbb{N}\right\}$ is a dense subset of $F$.
\end{enumerate}
We will refer to such $\left\{ \xi_{n}:n\in\mathbb{N}\right\}$ as a \textbf{measurable injective selection}.
\end{prop}

\begin{proof}
By the Kuratowski \& Ryll-Nardzewski's Selection Theorem (see \cite[Theorem (12.13)]{kechris2012classical}) there is a measurable selection: mappings $\theta_{n}:\mathbf{F}_{\tau}\left(X\right)\backslash\left\{ \emptyset\right\} \to X$, $n\in\mathbb{N}$, satisfying the first and the third properties. For each $n\in\mathbb{N}$, restrict the mapping $\theta_{n}$ to $\mathbf{F}_{\tau}^{\ast}\left(X\right)$ and modify it into the mapping $\xi_{n}:\mathbf{F}_{\tau}^{\ast}\left(X\right)\to X$ by letting
$$\xi_{n}\left(F\right)=\theta_{\pi_{n}\left(F\right)}\left(F\right),\quad F\in \mathbf{F}_{\tau}^{\ast}\left(X\right),$$
where $\pi_{n}:\mathbf{F}_{\tau}^{\ast}\left(X\right)\to\mathbb{N}$ is given by
$$\pi_{n}\left(F\right)=\inf\left\{ k\in\mathbb{N}:\#\left\{ \theta_{1}\left(F\right),\dotsc,\theta_{k}\left(F\right)\right\} =n\right\}.$$
Clearly, $\left\{ \xi_{n}:n\in\mathbb{N}\right\}$ forms an injective selection, and an elementary proof by induction on $\pi_{n}$ and $\xi_{n}$ shows that each $\xi_{n}$ is Effros-measurable.
\end{proof}

\begin{prop}
\label{Proposition: Effros random variables}
Let $\left(X,\mathcal{B}\right)$ be a standard Borel space. For every Polish topology $\tau$ that generates $\mathcal{B}$ there are mappings
$$\left\{\Xi_{A}:A\in\mathcal{B}\right\}\text{ of the form }\Xi_{A}:\mathbf{F}_{\tau}^{\ast}\left(X\right)\to\mathbb{Z}_{\geq 0}\cup\left\{\infty\right\},$$
with the following properties:
\begin{enumerate}
    \item $\Xi_{A\cup B}=\Xi_{A}+\Xi_{B}$ for every disjoint $A,B\in\mathcal{B}$.
    \item Each $\Xi_{A}$ is Effros-measurable, and together $\left\{\Xi_{A}:A\in\mathcal{B}\right\}$ generate the Effros σ-algebra of $\mathbf{F}_{\tau}^{\ast}\left(X\right)$.
    \item The identity $\Xi_{A}\left(F\right)=\#\left(A\cap F\right)$ holds in the following cases:\\
    (i) $A\in\tau$ (and every $F$).\\
    (ii) There is $O\in\tau$ such that $A\subset O\in\tau$ and $\#\left(F\cap O\right)<\infty$.\\
    (iii) $A\in\mathcal{B}$ and $F$ is $\tau$-discrete.
\end{enumerate}
\end{prop}

\begin{proof}
Pick a measurable injective selection $\left\{ \xi_{n}:n\in\mathbb{N}\right\}$ of $\mathbf{F}_{\tau}^{\ast}\left(X\right)$ as in Proposition \ref{Proposition: measurable injective selection}, and for $A\in\mathcal{B}$ put
$$\Xi_{A}\left(F\right)=\#\left\{ n\in\mathbb{N}:\xi_{n}\left(F\right)\in A\right\}=\sum\nolimits_{n\in\mathbb{N}}1_{\left\{ \xi_{n}\in A\right\} }\left(F\right).$$
Property (1) follows from the injectivity of the measurable injective selection. As for Property (2), the Effros-measurability of each $\Xi_{A}$ follows from the measurability of the measurable injective selection. We now make the observation that if $A\in\tau$ and $C\subset A$ is any countable set, then $\#\left(A\cap\overline{C}\right)=\#\left(A\cap C\right)$, where $\overline{C}$ denotes the closure of $C$, in the sense that they are either infinite together and otherwise both are finite and of the same cardinality. The same is true when $A\subset O\in\tau$ and $C\subset O$ is a finite set, for some $O\in\tau$. Thus, using the properties of measurable injective selection, this establishes Property (3)(i) and (ii), and in a straightforward way also Property (3)(iii). Finally, we note that by Property (3)(i) we have
$$B_{O}\cap\mathbf{F}_{\tau}^{\ast}\left(X\right)=\left\{ \Xi_{O}>0\right\} ,\quad O\in\tau,$$
and this completes the proof of the second part of Property (2).
\end{proof}

We can now prove Theorem \ref{Theorem: Poisson random set, revisited}. To this end we exploit the existence of the classical Poisson point process for locally compact Polish topologies as follows. Thus, let $\left(X,\mathcal{B},\mu\right)$ be a standard infinite measure space. Consider the classical construction of the Poisson point process: Following Proposition \ref{Proposition: classical Poisson point process}, pick a locally compact Polish topology $\vartheta$ for which $\mu$ is Radon, and let $\left(X_{\vartheta}^{\ast},\mathcal{B}_{\vartheta}^{\ast},\mu_{\vartheta}^{\ast}\right)$ with the random variables $N_{A}:\omega\mapsto\omega\left(A\right)$ for $A\in\mathcal{B}$. Consider the map
$$\Phi:X_{\vartheta}^{\ast}\to\mathbf{F}_{\vartheta}^{\ast}\left(X\right),\quad\Phi\left(\omega\right)=\mathrm{Supp}\left(\omega\right).$$
Since the Poisson point process consists of simple counting Radon measures, $\Phi$ is well-defined. It is injective and, since $\Phi^{-1}\left(B_{O}\right)=\left\{N_{O}>0\right\}$ for every $O\in\tau$, it is measurable. Define now
$$\Xi_{A}^{\vartheta}:\Phi\left(X_{\vartheta}^{\ast}\right)\subset \mathbf{F}_{\vartheta}^{\ast}\left(X\right)\to\mathbb{Z}_{\geq 0}\cup\left\{\infty\right\},\quad \Xi_{A}^{\vartheta}=N_A\circ\Phi^{-1},\quad A\in\mathcal{B}.$$
Evidently, the random variables $\left\{\Xi_{A}^{\vartheta}:A\in\mathcal{B}\right\}$ satisfies the condition of Proposition \ref{Proposition: Effros random variables}. Thus, the classical Poisson point process 
$$\left(X_{\vartheta}^{\ast},\mathcal{B}_{\vartheta}^{\ast},\mu_{\vartheta}^{\ast}\right)\text{ with }\left\{N_{A}:A\in\mathcal{B}\right\}$$
can be naturally identified via $\Phi$ with the Poisson random set
\begin{equation}
\label{eq:classicalPoissonset}
\left(\mathbf{F}_{\vartheta}^{\ast}\left(X\right),\mathcal{E}_{\vartheta}^{\ast}\left(X\right),\mu_{\vartheta}^{\ast}\right)\text{ with }\big\{\Xi_{A}^{\vartheta}:A\in\mathcal{B}\big\}.
\end{equation}

\begin{proof}[Proof of Theorem \ref{Theorem: Poisson random set}]
Let $\left(X,\mathcal{B},\mu\right)$ be a standard infinite measure space and by the assumption there is a Polish topology $\tau$ on $X$ for which $\mu\in\mathcal{M}_{\sigma}^{\tau}\left(X,\mathcal{B}\right)$. The proof will be divided into three part. In the first part we construct a generating algebra of sets $\mathbf{A}$ on $\mathbf{F}_{\tau}^{\ast}\left(X\right)$. In the second part we define $\mu_{\tau}^{\ast}$ by defining it as a pre-measure on $\mathbf{A}$ and extending using the Hahn-Kolmogorov Extension Theorem. In the third part we show the desired properties of $\mu_{\tau}^{\ast}$.

\subsubsection*{Part 1}

By the $\tau$-local-finiteness of $\mu$, there exists a countable base $\mathcal{O}=\left\{O_{1},O_{2},\dotsc\right\}$ for $\tau$ such that $\mu\left(O_{n}\right)<\infty$ for every $n$. Since $\tau$-closed sets are $G_{\delta}$-sets in $\tau$, by removing the largest $\tau$-open set which is $\mu$-null and restricting $\tau$ to the remaining $\tau$-closed set, we may assume that $\mu\left(O_{n}\right)>0$ for every $n$. For every $n$, let $\rho_{n}$ be the finest partition of $O_{1}\cup\dotsm\cup O_{n}$ that is generated by $\left\{O_{1},\dotsc,O_{n}\right\}$, hence every atom of $\rho_{n}$ is a $\tau$-open set with finite positive measure. Since $\mathcal{O}$ is a base for $\tau$, the ascending sequence of partitions $\left(\rho_{n}\right)_{n=1}^{\infty}$ converges to the partition into points of $X$, and the ascending sequence of σ-algebras $\left(\sigma\left(\rho_{n}\right)\right)_{n=1}^{\infty}$ generates $\mathcal{B}$ modulo $\mu$.

Fix $\left\{ \Xi_{A}:A\in\mathcal{B}\right\}$ as in Proposition \ref{Proposition: Effros random variables}. For every $A\in\mathcal{B}$ denote by $\Pi_{A}$ the partition of the measurable set $\left\{\Xi_{A}<\infty\right\}$, so that its atoms are $\left\{\Xi_{A}=k\right\}$, $k\in\mathbb{Z}_{\geq0}$. More generally, for every finite partition $\rho$ of a set $A\in\mathcal{B}$, denote by $\Pi_{\rho}$ the finest partition of $\left\{\Xi_{A}<\infty\right\}$ that is generated by the partitions $\Pi_{B}$, $B\in\rho$. For every $n$, recalling the partition $\rho_{n}$, we put
$$\Pi_{n}=\Pi_{\rho_{n}}.$$
Thus, $\rho_{n}$ is a finite partition of $O_{1}\cup\dotsm\cup O_{n}\subset X$ and $\Pi_{n}$ is an infinite partition of $\left\{\Xi_{O_{1}\cup\dotsm\cup O_{n}}<\infty\right\} \subset\mathbf{F}_{\tau}^{\ast}\left(X\right)$. Since the σ-algebra of $\mathbf{F}_{\tau}^{\ast}\left(X\right)$ is generated by $\left\{ \Xi_{O}:O\in\tau\right\}$, it follows that the ascending sequence of partitions $\left(\Pi_{n}\right)_{n=1}^{\infty}$ converges to the partition into points of $\mathbf{F}_{\tau}^{\ast}\left(X\right)$, and the ascending sequence of σ-algebras $\left(\sigma\left(\Pi_{n}\right)\right)_{n=1}^{\infty}$ generates the Effros σ-algebra of $\mathbf{F}_{\tau}^{\ast}\left(X\right)$. Thus, we obtain that
$$\mathbf{A}:=\sigma\left(\Pi_{1}\right)\cup\sigma\left(\Pi_{2}\right)\cup\dotsm$$
is an algebra of sets that generates its Effros σ-algebra  of $\mathbf{F}_{\tau}^{\ast}\left(X\right)$.

\subsubsection*{Part 2}

Observe that the atoms of each $\Pi_{n}$ are in one-to-one correspondence with all functions $\kappa:\rho_{n}\to\mathbb{Z}_{\geq0}$, $B\mapsto\kappa_{B}$, in that every such function corresponds to the nonempty atom
$${\textstyle \bigcap_{B\in\rho_{n}}}\left\{\Xi_{B}=\kappa_{B}\right\} \in\Pi_{n}.$$
To see why this is true, note that every atom $B\in\rho_{n}\subset\tau$ is a nonempty open set with $0<\mu\left(B\right)<\infty$, and thus $\mu\left(X\backslash B\right)=\infty$. Then pick a (closed) set $F'\subset B$ with $\#F'=\kappa_{B}$, and look at the closed set $F:=F'\cup\left(X\backslash B\right)$. By the properties of $\Xi_{B}$ we see that $\Xi_{B}\left(F\right)=\#\left(B\cap F\right)=\kappa_{B}$, namely $F\in\left\{\Xi_{B}=\kappa_{B}\right\}$. Since the atoms $\rho_{n}$ are disjoint, this shows that $\bigcap_{B\in\rho_{n}}\left\{\Xi_{B}=\kappa_{B}\right\}$ is nonempty for whatever choice of $\kappa$. 

Define $\mu_{0}^{\ast}$ on $\mathbf{A}$ as follows. For each $n$, we specify $\mu_{n}^{\ast}$ on $\Pi_{n}$ by letting  $\left\{ \Xi_{B},B\in\rho_{n}\right\} $ be independent and Poisson distributed with respective means $\mu\left(B\right)$, $B\in\rho_{n}$. Thus, we obtain a set function $\mu_{0}^{\ast}$ on $\mathbf{A}$ be letting $\mu_{0}^{\ast}\mid_{\sigma\left(\Pi_{n}\right)}=\mu_{n}^{\ast}$ for each $n$. It is not hard to show directly that $\mu_{0}^{\ast}$ is a consistent pre-measure on $\mathbf{A}$, but it would be shorter to utilize the existence of the classical Poisson point process with respect to some other Polish topology $\vartheta$ which is chosen to be locally compact and for which $\mu$ is a Radon, as is presented in \eqref{eq:classicalPoissonset}. First, for every $n$ and $\kappa:\rho_{n}\to\mathbb{Z}_{\geq 0}$ we have the obvious identity
\begin{equation}
    \label{eq: identity}
    \textstyle{\mu_{0}^{\ast}\big(\bigcap_{B\in\rho_{n}}\left\{ \Xi_{B}=\kappa_{B}\right\}\big)=\mu_{\vartheta}^{\ast}\big(\bigcap_{B\in\rho_{n}}\left\{ \Xi_{B}^{\vartheta}=\kappa_{B}\right\} \big)}
\end{equation}
We now observe that for each $n$, since $\Pi_{n}$ is a countable partition, $\sigma\left(\Pi_{n}\right)$ is nothing but all countable disjoint unions of the atoms of $\Pi_{n}$. Consequently, as the sequence of σ-algebras $\left(\sigma\left(\Pi_{n}\right)\right)_{n=1}^{\infty}$ is ascending, the algebra $\mathbf{A}$ consists of countable disjoint unions of atoms of the partitions $\left(\Pi_{n}\right)_{n=1}^{\infty}$. Then the family of identities \eqref{eq: identity} together with the fact that $\mu_{\vartheta}^{\ast}$ is a probability measure and in particular σ-additive, readily imply that $\mu_{0}^{\ast}$ is a consistent pre-measure on $\mathbf{A}$. Finally, by the Hahn-Kolmogorov Extension Theorem, $\mu_{0}^{\ast}$ extends to a genuine probability measure $\mu_{\tau}^{\ast}$ on $\mathbf{F}_{\tau}^{\ast}\left(X\right)$.

\subsubsection*{Part 3}

We now show that $\left\{\Xi_{A}:A\in\mathcal{B}\right\}$ forms a Poisson point process on the base space $\left(X,\mathcal{B},\mu\right)$. Let $A\in\mathcal{B}_{\mu}$ be arbitrary. For every $n$ denote
$$A_{n}=A\cap\left(O_{1}\cup\dotsm\cup O_{n}\right),$$
and define $\rho_{n}^{A}$ to be the finest partition of $O_{1}\cup\dotsm\cup O_{n}$ that is generated by $A_{n},O_{1},\dotsc,O_{n}$. While the sets $A_{n}$ may not be in $\tau$, since they are contained in $O_{1}\cup\dotsm\cup O_{n}$ we may apply Proposition \ref{Proposition: Effros random variables}(3)(ii). Thus, repeating the same construction as above when the partitions $\left(\rho_{n}\right)_{n=1}^{\infty}$ are replaced by the partitions $\left(\rho_{n}^{A}\right)_{n=1}^{\infty}$, we obtain a probability measure $\mu_{\tau}^{A}$ on $\mathbf{F}_{\tau}^{\ast}\left(X\right)$ with respect to which $\Xi_{A\cap O_{n}}$ is Poisson distributed with mean $\mu\left(A\cap O_{n}\right)$ for every $n$. Then by standard arguments it follows that $\Xi_{A}$ is Poisson distributed with mean $\mu\left(A\right)$ with respect to $\mu_{\tau}^{A}$. However, we evidently have that
$$\mu_{\tau}^{A}\left(B_{O_{n}}\right)=1-e^{-\mu\left(O_{n}\right)}=\mu_{\tau}^{\ast}\left(B_{O_{n}}\right),\quad n=1,2,\dotsc,$$
and the distribution of a random closed set in $\mathbf{F}_{\tau}\left(X\right)$ is determined by the probabilities of $B_{O}$, $O\in\mathcal{O}$ (see e.g. \cite[Theorem 1.3.20]{molchanov2017}, whose proof is valid also for a base for the topology), and it follows that $\mu_{\tau}^{A}=\mu_{\tau}^{\ast}$. This shows that $\left\{\Xi_{A}:A\in\mathcal{B}\right\}$ forms a Poisson point process with respect to $\mu_{\tau}^{\ast}$ and, moreover, $\mu_{\tau}^{\ast}$ is unique with respect to this property, establishing that the correspondence $\mathcal{M}_{\sigma}^{\tau}\left(X,\mathcal{B}\right)\to\mathcal{M}_{1}\left(\mathbf{F}_{\tau}^{\ast},\mathcal{E}_{\tau}^{\ast}\left(X\right)\right)$, $\mu\mapsto\mu_{\tau}^{\ast}$, is one-to-one.

Finally, in order to see that $\mu_{\tau}^{\ast}$ is supported on the class of $\tau$-discrete sets, note that ${\textstyle \bigcap_{n\geq 1}}\left\{ \Xi_{O_{n}}<\infty\right\}$ is an Effros-measurable subset of $\mathbf{F}_{\tau}^{\ast}\left(X\right)$ that consists only of $\tau$-discrete sets and, since $\mu_{\tau}^{\ast}\left(\Xi_{O_{n}}<\infty\right)=1$ for each $n\geq 1$, we deduce that it is a $\mu_{\tau}^{\ast}$-conull set.
\end{proof}

\subsection{Spatial Poisson suspensions}

In this section we prove Theorem \ref{Theorem: spatial Poisson suspension}.

Let $\left(X,\mathcal{B}\right)$ be a standard Borel space, and $\tau$ a Polish topology on $X$ that generates $\mathcal{B}$. For every $\tau$-homeomorphism $T$ of $X$ we define
$$T^{\ast}:\mathbf{F}_{\tau}^{\ast}\left(X\right)\to \mathbf{F}_{\tau}^{\ast}\left(X\right),\quad T^{\ast}\left(F\right):=T\left(F\right)=\left\{ T\left(x\right)\in X:x\in F\right\}.$$

\begin{lem}
\label{Lemma: T*1}
Let $\mu\in\mathcal{M}_{\sigma}\left(X,\mathcal{B}\right)$. If $T$ is a $\tau$-homeomorphism that preserves $\mu$, then $T^{\ast}$ is a transformation that preserves $\mu_{\tau}^{\ast}$ (as in Theorem \ref{Theorem: Poisson random set, revisited}).
\end{lem}

\begin{proof}
For every $O\in\tau$ we have
\begin{align*}
\mu_{\tau}^{\ast}\left(T^{\ast-1}\left(B_{O}\right)\right)
& = \mu^{\ast}\left(B_{T^{-1}\left(O\right)}\right)=\mu^{\ast}\left(\Xi_{T^{-1}\left(O\right)}>0\right)=1-e^{-\mu\left(T^{-1}\left(O\right)\right)} \\
& = 1-e^{-\mu\left(O\right)}=\mu^{\ast}\left(\Xi_{O}>0\right)=\mu_{\tau}^{\ast}\left(B_{O}\right),
\end{align*}
thus $\mu^{\ast}\circ T^{\ast-1}$ and  $\mu^{\ast}$ coincide on $B_{O}$, $O\in\tau$. Since the values on $B_{O}$, $O\in\tau$, determines random closed sets uniquely (see \cite[Theorem 1.3.20]{molchanov2017}) it follows that $T^{\ast}$ preserves $\mu^{\ast}$
\end{proof}

Note that $\left(T\circ S\right)^{\ast}=S^{\ast}\circ T^{\ast}$ whenever $T,S$ are $\tau$-homeomorphism, and in particular $T^{\ast}$ is invertible by $\left(T^{\ast}\right)^{-1}=\left(T^{-1}\right)^{\ast}$. Obviously, $\mathbf{F}_{\tau}^{\ast}\left(X\right)$ is $T^{\ast}$-invariant for whatever $\tau$-homeomorphism $T$ of $X$. Then from a $\tau$-Polish action $\mathbf{T}:G\curvearrowright\left(X,\mathcal{B}\right)$ we obtain the action
\begin{align}
\label{eq: Effros action}  \mathbf{T}^{\ast}:G\curvearrowright\left(\mathbf{F}_{\tau}^{\ast}\left(X\right),\mathcal{E}_{\tau}^{\ast}\left(X\right)\right),\quad \mathbf{T}^{\ast}:\left(g,F\right)\mapsto T_{g}^{\ast}\left(F\right).
\end{align}

\begin{lem}
\label{Lemma: T*2}
If $\mathbf{T}:G\curvearrowright\left(X,\mathcal{B}\right)$ is a $\tau$-Polish action then $\mathbf{T}^{\ast}:G\curvearrowright\left(\mathbf{F}_{\tau}\left(X\right),\mathcal{E}_{\tau}\left(X\right)\right)$, and in particular $\mathbf{T}^{\ast}:G\curvearrowright\left(\mathbf{F}_{\tau}^{\ast}\left(X\right),\mathcal{E}_{\tau}^{\ast}\left(X\right)\right)$ as in \eqref{eq: Effros action}, is a Borel action.
\end{lem}

\begin{rem}
In the literature this fact is viewed as elementary (see e.g. \cite[$\mathsection$ 2.4, Example (ii)]{beckerkechris1996}, \cite[$\mathsection$ 3.3]{gao2008invariant}). It is worth mentioning that when $\tau$ is locally compact (and only then), the {\it Fell topology} on $\mathbf{F}_{\tau}\left(X\right)$ is Polish and generates $\mathcal{E}_{\tau}\left(X\right)$ (see \cite[Exercise (12.7)]{kechris2012classical}). It can be verified, with the assistance of \cite[Theorem 4.8.6]{srivastava2008course}, \cite[Exercise (9.16) i)]{kechris2012classical}, that in this case $\mathbf{T}^{\ast}$ is further a Polish action in the Fell topology.
\end{rem}

\begin{proof}
Evidently $\mathbf{T}^{\ast}$ is an action and we verify that $\mathbf{T}^{\ast}$ is a Borel map by showing that $\mathbf{T}^{\ast-1}\left(B_{O}\right)$ is measurable in $G\times\mathbf{F}_{\tau}^{\ast}\left(X\right)$ for each $O\in\tau$. Using the Kuratowski \& Ryll-Nardzewski's Selection Theorem (see \cite[Theorem (12.13)]{kechris2012classical}), fix a measurable selection $\left\{\theta_{n}:n\in\mathbb{N}\right\}$ for $\mathbf{F}_{\tau}\left(X\right)$, that is a countable collection of Effros-measurable mappings $\theta_{n}:\mathbf{F}_{\tau}\left(X\right)\to X$ such that $\left\{\theta_{n}\left(F\right):n\in\mathbb{N}\right\}$ is a (countable) dense subset of $F$ for each $F\in\mathbf{F}_{\tau}^{\ast}\left(X\right)\backslash\left\{\emptyset\right\}$. Fix also a countable base $\mathcal{U}$ for $\tau$.

Let $O\in\tau$ be arbitrary. For each $g\in G$ and $F\in\mathbf{F}_{\tau}\left(X\right)$, using that $T_{g}^{-1}\left(O\right)$ is $\tau$-open,
\begin{align*}
    T_{g}\left(F\right)\cap O\neq\emptyset
    & \iff F\cap T_{g}^{-1}\left(O\right)\neq\emptyset\\
    &\iff\exists_{n\in\mathbb{N}}\left[\theta_{n}\left(F\right)\in T_{g}^{-1}\left(O\right)\right]\\
    & \iff\exists_{n\in\mathbb{N}}\exists_{U\in\mathcal{U}}\left[\left(U\subset T_{g}^{-1}\left(O\right)\right)\wedge\left(\theta_{n}\left(F\right)\in U\right)\right],
\end{align*}
so we may write
$$\mathbf{T}^{\ast-1}\left(B_{O}\right)=\bigcup_{n\in\mathbb{N}}\bigcup_{U\in\mathcal{U}}\left\{ g\in G:U\subset T_{g}^{-1}\left(O\right)\right\} \times\left\{ F\in\mathbf{F}_{\tau}\left(X\right):\theta_{n}\left(F\right)\in U\right\}.$$
Each set $\left\{ F\in\mathbf{F}_{\tau}\left(X\right):\theta_{n}\left(F\right)\in U\right\}$ is clearly Effros-measurable. In order to see the measurability of each $\left\{ g\in G:U\subset T_{g}^{-1}\left(O\right)\right\}$, write
\begin{align*}
    \left\{ g\in G:U\not\subset T_{g}^{-1}\left(O\right)\right\} 
    & = \left\{ g\in G:T_{g}\left(U\right)\cap X\backslash O\neq\emptyset\right\}\\
    & =\bigcup\nolimits_{x\in X\backslash O}\left\{ g\in G:x\in T_{g}\left(U\right)\right\}.
\end{align*}
Since for each $x\in X$ the map $G\to X$, $g\mapsto T_{g}^{-1}\left(x\right)$, is continuous, the set $\left\{ g\in G:x\in T_{g}\left(U\right)\right\} $ is open. This readily implies that $\left\{ g\in G:U\subset T_{g}\left(O\right)\right\}$ is closed hence measurable.
\end{proof}

\begin{proof}[Proof of Theorem \ref{Theorem: spatial Poisson suspension}]
Given a locally finite Polish action $\mathbf{T}:G\spa\left(X,\mathcal{B},\mu\right)$ with respect to a Polish topology $\tau$, using the construction of Theorem \ref{Theorem: Poisson random set, revisited}, the construction of the action as in \eqref{eq: Effros action}, together with Lemmas \ref{Lemma: T*1} and \ref{Lemma: T*2}, we obtain the spatial action 
$$\mathbf{T}^{\ast}:G\spa\left(\mathbf{F}_{\tau}^{\ast}\left(X\right),\mathcal{E}_{\tau}^{\ast}\left(X\right),\mu_{\tau}^{\ast}\right).$$
We complete the proof by showing that this is a Poissonian action, whose base action is $\mathbf{T}:G\spa\left(X,\mathcal{B},\mu\right)$, with respect to the Poisson point process $\left\{\Xi_{A}:A\in\mathcal{B}\right\}$ that is defined on $\left(\mathbf{F}_{\tau}^{\ast}\left(X\right),\mathcal{E}_{\tau}^{\ast}\left(X\right),\mu_{\tau}^{\ast}\right)$ as in Theorem \ref{Theorem: Poisson random set, revisited}.

Recall that by Proposition \ref{Proposition: Effros random variables}, for every $A\in\mathcal{B}$, we have that $\Xi_{A}\left(F\right)=\#\left(A\cap F\right)$ whenever $F\in\mathbf{F}_{\tau}^{\ast}\left(X\right)$ is $\tau$-discrete. Since $\mu_{\tau}^{\ast}$ is supported on the class of $\tau$-discrete sets, it follows that for $F$ in a $\mu_{\tau}^{\ast}$-conull set,
$$\Xi_{A}\left(T^{\ast}\left(F\right)\right)=\#\left(A\cap T^{\ast}\left(F\right)\right)=\#\left(T^{-1}\left(A\right)\cap F\right)=\Xi_{T^{-1}\left(A\right)}\left(F\right).$$
This readily shows that $\mathbf{T}$ is a base action for $\mathbf{T}^{\ast}$ as in Definition \ref{Definition: base action}, completing the proof of Theorem \ref{Theorem: spatial Poisson suspension}.
\end{proof}

\section{Constructing spatial actions from nonsingular spatial actions}

Here we aim to establish Theorem \ref{Theorem: spatial actions}. Recall the Polish group
$$\mathrm{Aut}\left(X,\mathcal{B},\left[\mu\right]\right)$$
of the (equivalence classes of) nonsingular transformations of a standard measure space $\left(X,\mathcal{B},\mu\right)$. It is worth mentioning that similarly to $\mathrm{Aut}\left(X,\mathcal{B},\mu\right)$, also $\mathrm{Aut}\left(X,\mathcal{B},\left[\mu\right]\right)$ is a L\'{e}vy group \cite[Theorem 6.1]{giordano2007some}, \cite[$\mathsection$4.5]{pestov2006}. Let us start with the important construction of the Maharam Extension, which allows one to realize the nonsingular transformations of one space as measure preserving transformations of another space.

\begin{prop}[Maharam Extension]
For every standard (finite or infinite) measure space $\left(X,\mathcal{B},\mu\right)$ there exists a standard infinite measure space $\big(\widetilde{X},\widetilde{\mathcal{B}},\widetilde{\mu}\big)$, with a continuous embedding of Polish groups
$$\mathrm{Aut}\left(X,\mathcal{B},\left[\mu\right]\right)\hookrightarrow\mathrm{Aut}\big(\widetilde{X},\widetilde{\mathcal{B}},\widetilde{\mu}\big),\quad T\mapsto\widetilde{T}.$$
\end{prop}

\begin{proof}
For a standard measure space $\left(X,\mathcal{B},\mu\right)$ define $\big(\widetilde{X},\widetilde{\mathcal{B}},\widetilde{\mu}\big)$ by
$$\widetilde{X}=X\times\mathbb{R},\quad\widetilde{\mathcal{B}}=\mathcal{B}\otimes\mathcal{B}\left(\mathbb{R}\right),\quad d\widetilde{\mu}\left(x,t\right)=d\mu\left(x\right)e^{t}dt.$$
Obviously, this is a standard infinite measure space. In order to define the desired embedding, let us denote the Radon-Nikodym cocycle by
$$\nabla:\mathrm{Aut}\left(X,\mathcal{B},\left[\mu\right]\right)\times X\to \left(0,\infty\right),\quad\nabla_{T}\left(\cdot\right)=\frac{d\mu\circ T^{-1}}{d\mu}\left(\cdot\right)\in L^{0}\left(X,\mathcal{B},\mu\right).$$
Note that this is only an {\it almost cocycle} in the sense that for every $T,S\in\mathrm{Aut}\left(X,\mathcal{B},\left[\mu\right]\right)$ it holds that
$$\nabla_{T\circ S}=\nabla_{T}\circ S\cdot\nabla_{S}\text{ in } L^{0}\left(X,\mathcal{B},\mu\right).$$
We now define the embedding $\mathrm{Aut}\left(X,\mathcal{B},\left[\mu\right]\right)\hookrightarrow\mathrm{Aut}\big(\widetilde{X},\widetilde{\mathcal{B}},\widetilde{\mu}\big)$ by
\begin{equation}
\label{eq: Maharam}T\mapsto\widetilde{T},\quad\widetilde{T}\left(x,t\right)=\left(T\left(x\right),t-\log\nabla_{T}\left(x\right)\right).
\end{equation}
It is straightforward to verify that this is a well-defined, continuous embedding of Polish groups.
\end{proof}

Suppose $G$ is a Polish group. A {\bf nonsingular} (Boolean) {\bf action} of $G$ on a standard measure space $\left(X,\mathcal{B},\mu\right)$ is a continuous (equivalently, measurable) homomorphism $\mathbf{T}:G\to\mathrm{Aut}\left(X,\mathcal{B},\left[\mu\right]\right)$. We denote such an action by
$$\mathbf{T}:G\curvearrowright\left(X,\mathcal{B},\left[\mu\right]\right).$$
A {\bf nonsingular spatial action} of $G$ on a standard measure space $\left(X,\mathcal{B},\mu\right)$ is a Borel action $\mathbf{T}:G\curvearrowright\left(X,\mathcal{B}\right)$, $\mathbf{T}:\left(g,x\right)\mapsto T_{g}\left(x\right)$, such that for every $g\in G$, $T_{g}$ is a nonsingular transformation of $\left(X,\mathcal{B},\mu\right)$. We denote such an action by
$$\mathbf{T}:G\spa\left(X,\mathcal{B},\left[\mu\right]\right).$$
Using the Maharam Extension construction, from every nonsingular action we obtain a (measure preserving) action, however this does not work in the spatial category since the Radon-Nikodym cocycle is merely an almost cocycle. When $G$ is locally compact, by the Mackey Cocycle Theorem the Radon-Nikodym cocycle admits a pointwise version ({\it strict version} in the terminology of \cite{becker2013cocycles}) and then the Maharam Extension does admit a point-realization. However, as it was shown by Becker \cite[Section E]{becker2013cocycles}, in general a strict version of the Radon-Nikodym cocycle may not exist.

\begin{proof}[Proof of Theorem \ref{Theorem: spatial actions}]
If $G$ admits a locally finite Polish action, then by Theorem \ref{Theorem: spatial Poisson suspension} the spatial Poisson suspension of this locally finite Polish action, with an appropriate Polish topology, is a finite spatial action. If $G$ admits a $\tau$-locally finite Polish nonsingular action $\mathbf{T}:G\spa\left(X,\mathcal{B},\left[\mu\right]\right)$, $\mathbf{T}:\left(g,x\right)\mapsto T_{g}\left(x\right)$, such that for each $g\in G$ the Radon-Nikodym derivatives $x\mapsto\nabla_{g}\left(x\right):=\nabla_{T_{g}}\left(x\right)$ is $\tau$-continuous, then the Maharam Extension
$$\widetilde{\mathbf{T}}:G\spa\big(\widetilde{X},\widetilde{\mathcal{B}},\widetilde{\mu}\big),\quad\widetilde{\mathbf{T}}:\left(g,\left(x,t\right)\right)\mapsto\left(T_{g}\left(x\right),t-\log\nabla_{g}\left(x\right)\right),$$
is a spatial action. In the coming discussion we omit the precise Polish topologies with respect to which the continuity properties are defined, as these are given by products of the obvious topologies on $G$, $X$ and $\mathbb{R}$. We further claim that this Maharam Extension is a locally finite Polish action. First, clearly the measure $\widetilde{\mu}$ is locally finite. In order to see that the action is Polish, namely that $\widetilde{\mathbf{T}}$ is a continuous map, we note that by the assumption it follows that $\widetilde{\mathbf{T}}$ is separately continuous, and by \cite[Theorem (9.14)]{kechris2012classical} it automatically follows that it is indeed jointly continuous. Having that $\widetilde{\mathbf{T}}$ is a locally finite Polish action with respect to $\widetilde{\tau}:=\tau\otimes\tau_{\mathbb{R}}$, where $\tau$ is the given Polish topology on $X$ and $\tau_{\mathbb{R}}$ is the usual topology of $\mathbb{R}$, again by Theorem \ref{Theorem: spatial Poisson suspension} we obtain that the spatial Poisson suspension
$$\widetilde{\mathbf{T}}^{\ast}:G\spa\big(\mathbf{F}_{\widetilde{\tau}}^{\ast}\big(\widetilde{X}\big),\mathcal{E}_{\widetilde{\tau}}^{\ast}\big(\widetilde{X}\big),\widetilde{\mu}_{\widetilde{\tau}}^{\ast}\big)$$
is a finite spatial action of $G$.
\end{proof}

\begin{proof}[Proof of Corollary \ref{Corollary: free spatial action}]
By \cite[Theorem 2.3]{kechris2022polish}, if $\mathbf{S}:G\spa\left(Y,\nu\right)$ is any finite spatial action that is faithful, the diagonal action on the infinite product $\mathbf{S}^{\mathbb{N}}:G\spa\left(Y^{\mathbb{N}},\nu^{\mathbb{N}}\right)$ is essentially free in the sense that there is a $\nu$-conull invariant subset of $Y^{\mathbb{N}}$ (where the invariance is pointwise) on which the action $\mathbf{S}$ is free. Evidently, this diagonal action is also a finite spatial action, so this completes the proof.
\end{proof}

\section{Diffeomorphism groups: classification and spatial actions}

Let $M$ be a compact smooth manifold. We will always assume that $M$ is $d$-dimensional Hausdorff connected and without boundary. Let a parameter $1\leq r\leq\infty$ be fixed, and denote by
$$\mathrm{Diff}^{r}\left(M\right)$$
the group of all $C^{r}$-diffeomorphisms from $M$ onto itself. It becomes a Polish group with the compact-open $C^{r}$-topology. Denote the connected component of the identity, as a normal subgroup, by
$$\mathrm{Diff}_{o}^{r}\left(M\right)\vartriangleleft\mathrm{Diff}^{r}\left(M\right).$$
It consists of the $C^{r}$-diffeomorphisms which are $C^{r}$-isotopic to the identity (see \cite[Corollary 1.2.2]{banyaga2013structure}). Diffeomorphism groups are locally connected (see \cite[Proposition 1.2.1]{banyaga2013structure}) so $\mathrm{Diff}_{o}^{r}\left(M\right)$ is a clopen subgroup of $\mathrm{Diff}^{r}\left(M\right)$, hence a non-locally compact Polish group in the subspace topology.

Observe that the local-connectedness of $\mathrm{Diff}^{r}\left(M\right)$ implies that it is a (not non-)Archimedean group. Indeed, non-Archimedean groups, admitting a base of clopen subgroups, are totally disconnected. Here we aim to show in the case of $r\neq d+1$ the stronger statement of Theorem \ref{Theorem: Diff}. To this end we exploit the following celebrated result that was established by Herman in the case when $M$ is a torus and $r=\infty$; by Thurston for every $M$ and $r=\infty$; and by Mather for every $M$ and $1\leq r<\infty$, $r\neq d+1$. See \cite[$\mathsection$2]{banyaga2013structure} and the references therein.

\begin{thm}[Herman, Thurston, Mather]
\label{Theorem: simplicity}
For all manifold $M$ as above, $\mathrm{Diff}_{o}^{r}\left(M\right)$ is a simple group whenever $r\neq d+1$.
\end{thm}

The second tool we need is the following characterization of groups of isometries of a locally compact metric space \cite[Theorem 1.2]{kwiatkowska2011spatial}.

\begin{thm}[Kwiatkowska \& Solecki]
\label{Theorem: Kwiatkowska Solecki}
A Polish group $G$ is a group of isometries of a locally compact metric space if and only if it possesses the following property:
\begin{quote}
Every identity neighborhood contains a closed subgroup $H$, such that $G/H$ is a locally compact space and the normalizer
$$N_{G}\left(H\right):=\left\{g\in G:gHg^{-1}=H\right\}\text{ is clopen.}$$
\end{quote}
\end{thm}

We can now prove Theorem \ref{Theorem: Diff}.

\begin{proof}[Proof of Theorem \ref{Theorem: Diff}]
Let $1\leq r\leq\infty$, $r\neq d+1$, be fixed. Since $\mathrm{Diff}_{o}^{r}\left(M\right)$ is a clopen subgroup of $\mathrm{Diff}_{o}^{r}\left(M\right)$, it is sufficient to show that $\mathrm{Diff}_{o}^{r}\left(M\right)$ is not a group of isometries of some locally compact metric space. Suppose otherwise toward a contradiction. Then by Theorem \ref{Theorem: Kwiatkowska Solecki} there exists a closed subgroup $H\leq\mathrm{Diff}_{o}^{r}\left(M\right)$ for which $\mathrm{Diff}_{o}^{r}\left(M\right)/H$ is a locally compact space and the normalizer $N_{\mathrm{Diff}_{o}^{r}\left(M\right)}\left(H\right)$ is clopen. Since $\mathrm{Diff}_{o}^{r}\left(M\right)/H$ is a locally compact space while $\mathrm{Diff}_{o}^{r}\left(M\right)$ is not, it follows that $H$, hence also $N_{\mathrm{Diff}_{o}^{r}\left(M\right)}\left(H\right)$, cannot be the trivial group. Since $N_{\mathrm{Diff}_{o}^{r}\left(M\right)}\left(H\right)$ is clopen and $\mathrm{Diff}_{o}^{r}\left(M\right)$ is connected, it follows that $N_{\mathrm{Diff}_{o}^{r}\left(M\right)}\left(H\right)=\mathrm{Diff}_{o}^{r}\left(M\right)$, namely $H$ is a normal subgroup of $\mathrm{Diff}_{o}^{r}\left(M\right)$. This contradicts Theorem \ref{Theorem: simplicity}.
\end{proof}

We move now to construct spatial actions of diffeomorphism groups.

\begin{proof}[Proof of Theorem \ref{Theorem: Diff spatial action}]
Let $M$ be a compact smooth manifold and pick a volume form $\mathrm{Vol}$ on $M$. In fact, the following construction works also when $M$ is merely paracompact and accordingly has an infinite total volume in $\mathrm{Vol}$; indeed, the compact-open $C^{r}$-topology on $\mathrm{Diff}^{r}\left(M\right)$ is still Polish and the Maharam Extension can be taken with respect to an infinite measure as well. Whether $M$ is compact or merely paracompact, look at the tautological action of the group $\mathrm{Diff}^{r}\left(M\right)$ on $M$ which is obviously nonsingular with respect to $\mathrm{Vol}$. Thus, we have a nonsingular locally finite Polish action
$$\mathbf{D}:\mathrm{Diff}^{r}\left(M\right)\spa\left(M,\mathrm{Vol}\right),\quad f.x=f\left(x\right).$$
The Radon-Nikodym cocycle of this action is the corresponding Jacobian, which is jointly continuous as a map $\mathrm{Diff}^{r}\left(M\right)\times M\to\left(0,\infty\right)$, hence by Theorem \ref{Theorem: spatial actions} we conclude that $\mathrm{Diff}^{r}\left(M\right)$ admits a nontrivial probability preserving spatial action. Using Corollary \ref{Corollary: free spatial action} we immediately obtain that there exists also such a free action.

In order to obtain also the ergodicity we look back at the steps in the construction as in the proof of Theorem \ref{Theorem: spatial actions}: starting with the locally finite Polish nonsingular action $\mathbf{D}$, we constructed the Maharam Extension $\widetilde{\mathbf{D}}$, from which we constructed the spatial Poisson suspension $\widetilde{\mathbf{D}}^{\ast}$ (with respect to the obvious Polish topology on $\widetilde{M}$). Then, in order to obtain a free action, we took the infinite diagonal product action $\big(\widetilde{\mathbf{D}}^{\ast}\big)^{\mathbb{N}}$. We then argue that the action $\big(\widetilde{\mathbf{D}}^{\ast}\big)^{\mathbb{N}}$ is ergodic. To this end we exploit the fact that the Maharam Extension $\widetilde{\mathbf{D}}$ is ergodic, which will be proved in Appendix \ref{Appendix: Maharam}, and then, as every infinite measure preserving action, the ergodicity of $\widetilde{\mathbf{D}}$ implies that it is null. From Theorem \ref{Theorem: ergodicity} it follows its Poisson suspension $\widetilde{\mathbf{D}}^{\ast}$ is weakly mixing. Recalling Remark \ref{Remark: weak mixing}(3), this readily implies that $\big(\widetilde{\mathbf{D}}^{\ast}\big)^{\mathbb{N}}$ is ergodic. This completes the proof of Theorem \ref{Theorem: Diff spatial action}.
\end{proof}

\subsection{Non-essentially countable orbit equivalence relations}

We introduce the necessary background to Corollary \ref{Corollary: Diff equivalence relations}. A general reference to the subject, with many references therein, is \cite{kechris2023theory}.

An equivalence relation $E$ on a standard Borel space $\left(X,\mathcal{B}\right)$ is a set $E\subset X\times X$ such that the condition $x\sim x'\iff\left(x,x'\right)\in E$ defines an equivalence relation on $X$. Such an equivalence relation is said to be {\bf Borel} if it is a Borel subset of $X\times X$, and {\bf countable} if each of its equivalence classes is countable. The Borel complexity of one equivalence relation relative to another can be tested by the possibility of producing a Borel reduction of the former into the latter in the following sense. A {\bf Borel reduction} of an equivalence relation $E$ on a standard Borel space $\left(X,\mathcal{B}\right)$ into an equivalence relation $F$ on a standard Borel space $\left(Y,\mathcal{C}\right)$, is a Borel function $f:X\to Y$ such that
$$\left(x,x'\right)\in E\iff \left(f\left(x\right),f\left(x'\right)\right)\in F\text{ for all }x,x'\in X.$$
An equivalence relation is said to be {\bf essentially countable} if it admits a Borel reduction into a countable Borel equivalence relation.

An important class of equivalence relations are {\bf orbit equivalence relations}. If $\mathbf{T}:G\curvearrowright\left(X,\mathcal{B}\right)$ is a Borel action of a Polish group $G$ on a standard Borel space $\left(X,\mathcal{B}\right)$, the associated orbit equivalence relation is
$$E_{\mathbf{T}}^{X}=\left\{\left(x,T_{g}\left(x\right)\right):x\in X,g\in G\right\}\subset X\times X.$$
It was shown by Kechris (see \cite[Theorem 4.10]{kechris1992countable}) that if $G$ is locally compact then $E_{\mathbf{T}}^{X}$ is always essentially countable, and he left as open question whether this property characterizes locally compact groups among the Polish groups (see \cite[Problem 1.2]{kechris2022polish}, \cite[Problem 4.16]{kechris2023theory} and the references therein).

\begin{proof}[Proof of Corollary \ref{Corollary: Diff equivalence relations}]
By Theorem \ref{Theorem: Diff spatial action}, every diffeomorphism group admits a free spatial action on a standard probability space. If the orbit equivalence relation of this action would be essentially countable, then by a theorem of Feldman \& Ramsay \cite[Theorem A]{feldman1985countable} (c.f. \cite[Theorem 2.1]{kechris2022polish}) it would follow that the acting group is locally compact, which is false, hence this orbit equivalence relation is non-essentially countable.
\end{proof}

\section{Open problems}

We revisit the question of Glasner, Tsirelson \& Weiss \cite[Question 1.2]{glasnertsirelsonweiss2005} of whether L\'{e}vy groups admit nonsingular spatial action. In their proof of the non-existence of probability preserving actions of L\'{e}vy groups, it was sufficient to show the non-existence of such \emph{Polish} actions. This reduction was possible thanks to a theorem of Becker \& Kechris \cite[$\mathsection$ 5.2]{beckerkechris1996}, by which every Borel action of a Polish group admits a Polish topology with respect to which it becomes a Polish action. In locally finite Polish actions, by definition, there exists such a topology that, simultaneously, makes the action Polish and the measure locally finite. It is then natural to ask about the following refinements of Becker \& Kechris' theorem:

\begin{question}
\label{Question: Radon model}
Let $\mathbf{T}:G\curvearrowright\left(X,\mathcal{B}\right)$ be a Borel action of a Polish group $G$, preserving an infinite measure $\mu$. When does there exist a Polish topology on $X$ with respect to which, simultaneously, $\mathbf{T}$ is Polish and $\mu$ is locally finite?
\end{question}

Diffeomorphism groups were shown, in Theorems \ref{Theorem: Diff} and \ref{Theorem: Diff spatial action}, to belong neither to the class of groups of isometries of a locally compact metric space nor the class of L\'{e}vy groups. Thus it is natural to ask:

\begin{question}
Do diffeomorphism groups possess the Mackey property?
\end{question}

The following related question was posed by Moore \& Solecki \cite[$\mathsection$4]{moore2012boolean}. For a compact smooth manifold $M$, let $C^{\infty}\left(M,S^{1}\right)$ be the group of $C^{\infty}$-functions from $M$ to the unit circle $S^{1}$ with pointwise multiplication and the compact-open $C^{\infty}$-topology. It was shown by Moore \& Solecki \cite[Theorem 1.1]{moore2012boolean} that the Mackey property fails for continuous functions from $M$ to $S^{1}$, and they ask about the Mackey property for $C^{\infty}\left(M,S^{1}\right)$, describing it as 'tempting to conjecture'.

The Mackey property of a Polish group refers to all of its actions at once. However, it is possible for a Polish group without the Mackey property to admit spatial actions. Indeed, the homeomorphism group of the circle $S^{1}$ does not possess the Mackey property by the aforementioned result of Moore \& Solecki, and yet it acts spatially and ergodically on $S^{1}$ with its Lebesgue measure by $f.z=f\left(1\right)z$. This suggests to deviate from the general Mackey property and raise the following problem that seems to be widely open:

\begin{problem}
Consider the following actions a Polish group may possess:\footnote{The notations were borrowed from the Krieger-type of countable groups actions, which in turn were borrowed from the classification of factors in von Neumann algebras.}
\begin{description}
    \item[$\bullet$ Type $\mathrm{II}_{1}$] Spatial ergodic probability preserving actions.
    \item[$\bullet$ Type $\mathrm{II}_{\infty}$] Spatial ergodic infinite measure preserving actions.
    \item[$\bullet$ Type $\mathrm{III}$] Spatial ergodic nonsingular actions without an absolutely continuous invariant measure.
\end{description}
Beyond the locally compact case, little is known about whether a given Polish group admits an action of Type $\mathrm{II}_{1}$, $\mathrm{II}_{\infty}$ or $\mathrm{III}$. A particular important case is Glasner, Tsirelson \& Weiss' question: L\'{e}vy groups do not admit actions of Type $\mathrm{II}_{1}$, but do they admit actions of Type $\mathrm{III}$? What about Type $\mathrm{II}_{\infty}$?\\
There are also classification aspects of this problem: does the admission of an action of a certain type imply the admission of an action of some other type? On the contrary, is there a Polish group that admits an action of a certain type but not an action of some other type?
\end{problem}

\begin{rem}
We comment on known results in the above problem
    \begin{enumerate}
        \item Groups of isometries of locally compact metric spaces admit actions of Types $\mathrm{II}_{\infty}$ (there is a Haar measure) and $\mathrm{II}_{1}$ (the Poisson suspension of the Haar measure, which is spatial by \cite[Theorem 1.1]{kwiatkowska2011spatial} or by Theorem \ref{Theorem: spatial Poisson suspension}). See \cite{danilenko2023krieger} with regard to Type $\mathrm{III}$.
        \item Diffeomorphism groups of compact smooth manifolds admit actions of Types $\mathrm{III}$ (tautologically), $\mathrm{II}_{\infty}$ (the Maharam Extension) and $\mathrm{II}_{1}$ (the spatial Poisson suspension as in Theorem \ref{Theorem: Diff spatial action}).
        \item The group $\mathrm{Homeo}_{+}\left(\left[0,1\right]\right)$ of orientation preserving homeomorphisms admits no Boolean action whatsoever by a theorem of Megrelishvili \cite[Theorem 3.1]{megrelishvili2001every} (see also \cite[Remark 1.7]{glasnertsirelsonweiss2005}).
        \item The Maharam Extension construction demonstrates that in certain cases, the admission of a Type $\mathrm{III}$ action implies the admission of a Type $\mathrm{II}_{\infty}$ action. The Poisson suspension construction demonstrates that in certain cases, the admission of a Type $\mathrm{II}_{\infty}$ action implies the admission of a Type $\mathrm{II}_{1}$ action. By Theorem \ref{Theorem: spatial actions}, these implications are valid when the actions possess appropriate continuity properties.
    \end{enumerate}
\end{rem}

\begin{appendix}

\section{The First Chaos of Poisson point processes}
\label{Appendix: The First Chaos of a Poisson Point process}

Fix a generative Poisson point process $\mathcal{P}=\left\{P_{A}:A\in\mathcal{B}\right\}$ that is defined on $\left(\Omega,\mathcal{F},\mathbb{P}\right)$ with base space $\left(X,\mathcal{B},\mu\right)$ as in Definition \ref{Definition: Poisson point process}. In the following we describe some fundamental properties of the {\bf first chaos} of $\mathcal{P}$, namely the real Hilbert space
$$H_{1}\left(\mathcal{P}\right):=\overline{\mathrm{span}}\left\{ P_{A}-\mu\left(A\right):A\in\mathcal{B}_{\mu}\right\} \subset L_{\mathbb{R}}^{2}\left(\Omega,\mathcal{F},\mathbb{P}\right).$$

\subsubsection{Fock space and Chaos decomposition}
\label{Appendix: Chaos decomposition}

The (real, symmetric) {\bf Fock space} associated with a Hilbert space $\mathcal{H}$ is, by definition, the Hilbert space
$$\mathrm{F}\left(\mathcal{H}\right)=\bigoplus\nolimits_{n=0}^{\infty}\mathcal{H}^{\odot n},$$
where $\mathcal{H}^{\odot 0}=\mathbb{R}$ and $\mathcal{H}^{\odot n}$ for $n\geq 1$ is the symmetric $n^{\mathrm{th}}$ tensor product of $\mathcal{H}$, i.e. its elements are the vectors in the usual tensor product $\mathcal{H}^{\otimes n}$ that are invariant to permutations of their coordinates, which are generated by elements of the form
$$u_{1}\odot\dotsm\odot u_{n}:=\frac{1}{n!}\sum_{\sigma\in\mathrm{Sym}\left(n\right)}u_{\sigma\left(1\right)}\otimes\dotsm\otimes u_{\sigma\left(n\right)},$$
and with the inner product that is given by
$$\left\langle u_{1}\odot\dotsm\odot u_{n},v_{1}\odot\dotsm\odot v_{n}\right\rangle _{\mathcal{H}^{\odot n}}:=\frac{1}{n!}\sum\nolimits_{\sigma\in\mathrm{Sym}\left(n\right)}\prod\nolimits_{j=1}^{n}\left\langle u_{j},v_{\sigma\left(j\right)}\right\rangle _{\mathcal{H}}.$$
Here the operation $\bigoplus$ denotes the operation of taking the Hilbert space obtained as the completion of the direct sum.

For the classical Poisson point process $\mathcal{N}$ that is defined on $\left(X^{\ast},\mathcal{B}^{\ast},\mu^{\ast}\right)$ with the base space $\left(X,\mathcal{B},\mu\right)$, the Fock space decomposition refers to the isometric isomorphism between $L_{\mathbb{R}}^{2}\left(X^{\ast},\mathcal{B}^{\ast},\mu^{\ast}\right)$ and $\mathrm{F}\left(L_{\mathbb{R}}^{2}\left(X,\mathcal{B},\mu\right)\right)$. See \cite{liebscher1994isomorphism} (see also \cite[$\mathsection$18]{last2017lectures}). Then for the given $\mathcal{P}$ as in Definition \ref{Definition: Poisson point process}, using Proposition \ref{Proposition: Poisson point process} we obtain an isometric isomorphism between $L_{\mathbb{R}}^{2}\left(\Omega,\mathcal{F},\mathbb{P}\right)$ and $\mathrm{F}\left(L_{\mathbb{R}}^{2}\left(X,\mathcal{B},\mu\right)\right)$. Thus, the Hilbert space $L_{\mathbb{R}}^{2}\left(\Omega,\mathcal{F},\mathbb{P}\right)$ has the {\bf chaos decomposition} into
\begin{equation}
\label{eq: chaos}
L_{\mathbb{R}}^{2}\left(\Omega,\mathcal{F},\mathbb{P}\right)=\bigoplus\nolimits_{n=0}^{\infty}H_{n}\left(\mathcal{P}\right),
\end{equation}
each $H_{n}\left(\mathcal{P}\right)$ is called the \emph{$n^{\mathrm{th}}$ chaos} with respect to $\mathcal{P}$. The following description of the chaos structure for an abstract $\mathcal{P}$ requires stochastic integration against $\mathcal{P}$ as a random measure, which is justified in Proposition \ref{Proposition: Poisson point process}.

\begin{itemize}
    \item $H_{0}\left(\mathcal{P}\right)=\mathbb{R}$.
    \item $H_{1}\left(\mathcal{P}\right):=\overline{\mathrm{span}}\left\{ P_{A}-\mu\left(A\right):A\in\mathcal{B}_{\mu}\right\} \subset L_{\mathbb{R}}^{2}\left(\Omega,\mathcal{F},\mathbb{P}\right)$.
    \item $\dotsm$
    \item $H_{n}\left(\mathcal{P}\right)=\overline{\mathrm{span}}\left\{ \int_{X^{\odot n}}1_{A}^{\odot n}d\left(\mathcal{P}-\mu\right)^{\odot n}:A\in\mathcal{B}_{\mu}\right\}\subset L_{\mathbb{R}}^{2}\left(\Omega,\mathcal{F},\mathbb{P}\right)$, where $X^{\odot n}$ is the set of sequences in $X^{n}$ consisting of $n$ disjoint elements.
\end{itemize}

With this chaos decomposition of $L_{\mathbb{R}}^{2}\left(\Omega,\mathcal{F},\mathbb{P}\right)$, the isometric isomorphism
$$I_{\mu}:L_{\mathbb{R}}^{2}\left(\Omega,\mathcal{F},\mathbb{P}\right)\to\mathrm{F}\left(L_{\mathbb{R}}^{2}\left(X,\mathcal{B},\mu\right)\right),$$
is given by stochastic integration against $\mathcal{P}$ as follows.
\begin{itemize}
    \item $I_{\mu}:\mathbb{R}f_{0}\to H_{0}\left(\mathcal{P}\right)$ is given by $I_{\mu}:cf_{0}\mapsto c$, where $\mathbb{R}f_{0}$ is a one-dimensional space that is spanned by a distinguish norm-one vector $f_{0}\in L_{\mathbb{R}}^{2}\left(X,\mathcal{B},\mu\right)$, called {\it vacuum vector}, whose specification will be unimportant for us.
    \item $I_{\mu}:L_{\mathbb{R}}^{2}\left(X,\mathcal{B},\mu\right)\to H_{1}\left(\mathcal{P}\right)$ is given by the stochastic integral
    $$I_{\mu}:1_{A}\mapsto \int_{X}1_{A}d\left(\mathcal{P}-\mu\right)=P_{A}-\mu\left(A\right),\quad A\in\mathcal{B}_{\mu}.$$
    \item $\dotsm$
    \item $I_{\mu}:L_{\mathbb{R}}^{2}\left(X,\mathcal{B},\mu\right)^{\odot n}\to H_{n}\left(\mathcal{P}\right)$ is given by the stochastic integral
    $$I_{\mu}:1_{A}^{\odot n}\mapsto\int_{X^{\odot n}}1_{A}^{\odot n}d\left(\mathcal{P}-\mu\right)^{\odot n},\quad A\in\mathcal{B}_{\mu}.$$
\end{itemize}

In general, every unitary operator $U$ of a Hilbert space $\mathcal{H}$ induces a unitary operator $\mathrm{F}\left(U\right)$ of the Fock space $\mathrm{F}\left(\mathcal{H}\right)$, that is defined by letting $\mathrm{F}\left(U\right)$ acts on $\mathcal{H}^{\odot n}$ as the $n^{\mathrm{th}}$ tensor product $U^{\otimes n}$. Thus, for every $T\in\mathrm{Aut}\left(X,\mathcal{B},\mu\right)$, by looking at the Koopman operator $U_{T}$ we obtain the operator $\mathrm{F}\left(U_{T}\right)$ of $\mathrm{F}\left(L_{\mathbb{R}}^{2}\left(X,\mathcal{B},\mu\right)\right)\cong L_{\mathbb{R}}^{2}\left(\Omega,\mathcal{F},\mathbb{P}\right)$. This operator can also be described without reference to the Fock space structure: if $S\in\mathrm{Aut}\left(\Omega,\mathcal{F},\mathbb{P}\right)$ is a Poissonian transformation with base transformation $T$ (for instance, $S=T^{\ast}$ as in Proposition \ref{Proposition: classical Poisson point process}), we obtain the Koopman operator $U_{S}$ of $L_{\mathbb{R}}^{2}\left(\Omega,\mathcal{F},\mathbb{P}\right)$. The equivariance property that relates $T$ and $S$ as in Definition \ref{Definition: base action} implies that
$$\mathrm{F}\left(U_{T}\right)=U_{S}.$$

\subsubsection{Infinitely Divisible distributions and the First Chaos}

We start by recalling some of the basics of infinitely divisible distributions. A general reference to the subject is \cite{sato1999levy}.

A distribution is {\bf infinitely divisible} if for every positive integer $n$ it can be presented as the $n^{\mathrm{th}}$-power convolution of some other distribution. By the fundamental L\'{e}vy--Khintchine Representation (see \cite[Chapter 2, $\mathsection$8]{sato1999levy}), every infinitely divisible distribution is completely determined by a triplet
$$\left(\sigma,\ell,b\right),\text{ where }\sigma\geq0,\,b\in\mathbb{R}\text{ and }\ell\text{ is a {\bf L\'{e}vy measure}}.$$
By definition, a L\'{e}vy measure is a σ-finite Borel (possibly with atoms) measure $\ell$ on $\mathbb{R}$ with
$$\ell\left(\left\{0\right\}\right)=0\text{ and }\int_{\mathbb{R}}t^{2}\wedge 1 d\ell\left(t\right)<\infty.$$
According to the L\'{e}vy--Khintchine Representation, a random variable $W$ has infinitely divisible distribution associated with a triplet $\left(\sigma,\ell,b\right)$ if its characteristic function takes the form
$$\mathbb{E}\left(\exp\left(itW\right)\right)=\exp\big(-t^{2}\sigma^{2}/2+\int_{\mathbb{R}}\left(e^{itx}-1-itx\cdot1_{\left[-1,+1\right]}\left(x\right)\right)d\ell\left(x\right)+itb\big).$$
Thus, every infinitely divisible distribution is the convolution of a centred Gaussian distribution with variance $\sigma^{2}$, and another infinitely divisible distribution that is determined by a L\'{e}vy measure $\ell$ and a constant $b$ via the L\'{e}vy--Khintchine Representation. An infinitely divisible distribution for which the Gaussian part vanishes, namely $\sigma=0$, is referred to as {\bf infinitely divisible Poissonian} distribution, henceforth {\bf {\sc id}p}. The fundamental examples of {\sc id}p distributions are Poisson distributions and, more generally, compound Poisson distributions. We refer to a random variable as {\sc id}p if its distribution is {\sc id}p.

Given a generative Poisson point process $\mathcal{P}$, the restriction of the Poisson stochastic integral in the aforementioned chaos decomposition to the first chaos $H_{1}\left(\mathcal{P}\right)$, forms an isometric isomorphism of the Hilbert spaces $$I_{\mu}:L_{\mathbb{R}}^{2}\left(X,\mathcal{B},\mu\right)\xrightarrow{\sim}H_{1}\left(\mathcal{P}\right).$$
We introduce some useful properties of this stochastic integral.

\begin{prop}
\label{Proposition: First chaos}
In the above setting the following hold.
\begin{enumerate}
    \item For every $f\in L_{\mathbb{R}}^{1}\left(X,\mathcal{B},\mu\right)\cap L_{\mathbb{R}}^{2}\left(X,\mathcal{B},\mu\right)$,
    $$I_{\mu}\left(f\right)=\int_{X}fd\mathcal{P}-\int_{X}fd\mu.$$
    \item For every $f\in L_{\mathbb{R}}^{2}\left(X,\mathcal{B},\mu\right)$, $I_{\mu}\left(f\right)$ is an {\sc id}p random variable whose L\'{e}vy measure is given by
    $$\ell_{f}:=\mu\mid_{\left\{f\neq 0\right\}}\circ f^{-1}.$$
    \item For every $f\in L_{\mathbb{R}}^{2}\left(X,\mathcal{B},\mu\right)$ for which $I_{\mu}\left(f\right)$ is bounded from below,
    $$\ell_{f}\left(\mathbb{R}_{<0}\right)=0\text{ and }\int_{\mathbb{R}\geq0}td\ell_{f}\left(t\right)<\infty.$$    
\end{enumerate}
\end{prop}

\begin{proof}
The stochastic integral $I_{\mu}$ is generally defined on the dense subspace $L_{\mathbb{R}}^{1}\left(X,\mathcal{B},\mu\right)\cap L_{\mathbb{R}}^{2}\left(X,\mathcal{B},\mu\right)$ as in part (2), and extends to $L_{\mathbb{R}}^{2}\left(X,\mathcal{B},\mu\right)$ by continuity. It is a classical fact that if $f\in L_{\mathbb{R}}^{2}\left(X,\mathcal{B},\mu\right)$ then $I_{\mu}\left(f\right)$ is an {\sc id}p random variable and its L\'{e}vy measure is $\ell_{f}$ as in part (1) (see e.g. \cite[Lemma 20.6]{sato1999levy}, \cite[Proposition 2.10]{rosinki2018representations} and note that on the dense subspace $L_{\mathbb{R}}^{1}\left(X,\mathcal{B},\mu\right)\cap L_{\mathbb{R}}^{2}\left(X,\mathcal{B},\mu\right)$ the stochastic integral $I_{\mu}$ differs from the stochastic integrals in these references only by a constant, whence they have the same L\'{e}vy measure). This establishes parts (1) and (2). 

In order to establish part (3) we exploit the general characterization of {\sc id}p random variables that are bounded from below as in \cite[Theorem 24.7]{sato1999levy}, by which if $I_{\mu}\left(f\right)$ is bounded from below then its L\'{e}vy measure $\ell_{f}$ satisfies the following two properties. First, $\ell_{f}\left(\mathbb{R}_{<0}\right)=0$ as in the first property in part (3). Second, one of the following alternatives occurs ({\it type A} or {\it type B} in the terminology of \cite[Definition 11.9]{sato1999levy}):
\begin{enumerate}
    \item $\ell\left(\mathbb{R}_{\geq 0}\right)<\infty$, in which case by the Cauchy-Schwartz inequality 
    $$\int_{\mathbb{R}_{\geq0}}td\ell_{f}\left(t\right)=\int_{\left\{ f> 0\right\} }fd\mu\leq\ell_{f}\left(\mathbb{R}_{\geq0}\right)\int_{\left\{ f>0\right\} }f^{2}d\mu<\infty.$$
    \item $\ell_{f}\left(\mathbb{R}_{\geq 0}\right)=\infty$ and $\int_{\left\{0\leq t\leq 1\right\}}td\ell_{f}\left(t\right)<\infty$, in which case we have
    $$\int_{\mathbb{R}_{\geq0}}td\ell_{f}\left(t\right)=\int_{\left\{ 0\leq t\leq1\right\} }td\ell_{f}\left(t\right)+\int_{\left\{ t>1\right\} }td\ell_{f}\left(t\right)<\infty,$$
    where the finiteness of the second term follows from
    $$\int_{\left\{ t>1\right\} }td\ell_{f}\left(t\right)=\int_{\left\{ f>1\right\} }fd\mu\leq\int_{\left\{ f>1\right\} }f^{2}d\mu<\infty.$$
\end{enumerate}
This completes the proof of part (3).
\end{proof}

\section{Ergodicity of Maharam Extensions}
\label{Appendix: Maharam}

A useful criteria for the ergodicity of the Maharam Extension is based on Krieger's theory of orbit equivalence classification of nonsingular transformations (i.e. nonsingular actions of $\mathbb{Z}$). While this theory, by its nature, applies to countable amenable groups, by reviewing its details one may derive a general criterion to the ergodicity of Maharam Extensions for actions of general groups. In the following discussion, an action of a general group $G$ is a homomorphism from $G$ into $\mathrm{Aut}\left(X,\mathcal{B},\left[\mu\right]\right)$, and the measurability or continuity of this homomorphism are irrelevant.

Let $G$ be a group. Suppose we are given a nonsingular action of $G$ on a standard measure space $\left(X,\mathcal{B},\mu\right)$, that is a homomorphism
$$\mathbf{T}:G\to\mathrm{Aut}\left(X,\mathcal{B},\left[\mu\right]\right),\quad\mathbf{T}:\left(g,x\right)\mapsto T_{g}\left(x\right).$$
Denote its Radon-Nikodym cocycle by 
$$\nabla_{g}\left(\cdot\right)=\frac{d\mu\circ T_{g}}{d\mu}\left(\cdot\right)\in L^{1}\left(X,\mathcal{B},\mu\right),\quad g\in G.$$
Let $\eta$ be the measure on $\mathbb{R}$ that is defined by $d\eta\left(t\right)=e^{t}dt$, where $dt$ denotes the Lebesgue measure on $\mathbb{R}$, and put
$$\widetilde{X}=X\times\mathbb{R},\,\widetilde{\mathcal{B}}=\mathcal{B}\otimes\mathcal{B}\left(\mathbb{R}\right)\text{ and }\widetilde{\mu}=\mu\otimes\eta.$$
The Maharam Extension of $\mathbf{T}$ is the infinite measure preserving action
$$\widetilde{\mathbf{T}}:G\to\mathrm{Aut}\left(\widetilde{X},\widetilde{\mathcal{B}},\widetilde{\mu}\right),\quad\widetilde{T}_{g}:\left(x,t\right)\mapsto\left(T_{g}\left(x\right),t-\log\nabla_{g}\left(x\right)\right).$$

A number $s\in\mathbb{R}$ is said to be an {\bf essential value} of $\mathbf{T}$ if:
\begin{quote}
    Given any $A\in\mathcal{B}$ with $\mu\left(A\right)>0$ and any $\epsilon>0$, there exist $A\supset B\in\mathcal{B}$ with $\mu\left(B\right)>0$ and some $g\in G$, such that $T_{g}\left(B\right)\subset A$ and $\log\nabla_{g}\left(B\right)\subset\left(s-\epsilon,s+\epsilon\right)$.
\end{quote}
The set of all essential values of $\mathbf{T}$ is called the {\bf ratio set} and is denoted by
$$\mathrm{r}\left(\mathbf{T},\mu\right),\text{ or also }\mathrm{r}\left(\mathbf{T}\right)\text{ when }\mu\text{ is understood}.$$
The ratio set is a closed subgroup of $\mathbb{R}$ depending only on the measure class of $\mu$, and as such it is a principle invariant in Krieger's theory \cite[$\mathsection$ 3]{schmidt1976}.

\begin{prop}[Schmidt]
\label{Proposition: Schmidt}
Suppose $\mathbf{T}$ is an ergodic nonsingular action. If $\mathrm{r}\left(\mathbf{T}\right)=\mathbb{R}$ then the Maharam Extension $\widetilde{\mathbf{T}}$ is ergodic.
\end{prop}

\begin{proof}
Let $A\in\widetilde{\mathcal{B}}$ be a $\widetilde{\mathbf{T}}$-invariant set. Denote by $\left(S_{s}\right)_{s\in\mathbb{R}}$ the flow on $\widetilde{X}$ given by $S_{s}\left(x,t\right)=\left(x,t+s\right)$. Since $\mathrm{r}\left(\mathbf{T}\right)=\mathbb{R}$, by \cite[Theorem 5.2]{schmidt1976}\footnote{Schmidt's theorem is formulated when $G$ is countable and $\eta$ is the Lebesgue measure. The proof of the part that is being used here remains valid for every group. Also, since $\eta$ is mutually absolutely continuous with the Lebesgue measure, the choice between those measures does not affect the ergodicity of the Maharam extension.} it follows that
$$\widetilde{\mu}\left(A\triangle S_{s}\left(A\right)\right)=0\text{ for every }s\in\mathbb{R},$$
so that $A$ is $\left(S_{s}\right)_{s\in\mathbb{R}}$-invariant. By the Fubini Theorem there are $A_{1}\in\mathcal{B}$ and $A_{2}\in\mathcal{B}\left(\mathbb{R}\right)$ such that $\widetilde{\mu}\left(A\triangle\left(A_1\times A_{2}\right)\right)=0$ and $A_{2}$ is invariant to all the translations of $\mathbb{R}$, hence either $\eta\left(A_{2}\right)=0$ or $\eta\left(\mathbb{R}\backslash A_{2}\right)=0$. Since $A$ is $\widetilde{\mathbf{T}}$-invariant it follows that $A_{1}$ is $\mathbf{T}$-invariant, and from the ergodicity of $\mathbf{T}$ it follows that either $\mu\left(A_{1}\right)=0$ or $\mu\left(X\backslash A_{1}\right)=0$. Thus, either $\widetilde{\mu}\left(A\right)=0$ or $\widetilde{\mu}\left(\left(X\times\mathbb{R}\right)\backslash A\right)=0$, completing the proof.
\end{proof}

\begin{prop}
Let $M$ be a paracompact smooth manifold with a volume form $\mathrm{Vol}$, let $1\leq r\leq\infty$, and denote by $\mathbf{D}:\mathrm{Diff}^{r}\left(M\right)\to\mathrm{Aut}\left(M,\left[\mathrm{Vol}\right]\right)$ the tautological nonsingular action. Then the Maharam Extension $\widetilde{\mathbf{D}}:\mathrm{Diff}^{r}\left(M\right)\to\mathrm{Aut}\big(\widetilde{M},\widetilde{\mathrm{Vol}}\big)$ is ergodic.
\end{prop}

\begin{rem}
While this proposition requires a proof which turns to be somehow technical, it should be regarded as easy. In fact, in common cases much more is known: when $M$ is compact and the dimension is either $d=1$ (Katznelson) or $d\geq 3$ (Herman), there exists a single diffeomorphism whose Maharam Extension is ergodic. See the introduction of \cite[$\mathsection$9]{hawkins2021ergodic} and the references therein.
\end{rem}

\begin{proof}
Note that $\mathbf{D}$ acts transitively on $M$ (for $d=1$ it is trivial, and for $d\geq2$ see \cite[Lemma 2.1.10]{banyaga2013structure}). Then every function $\phi\in L^{1}\left(M,\mathrm{Vol}\right)$ that satisfies $\phi\circ f=\phi$ for every $f\in\mathrm{Diff}^{r}\left(M\right)$ is constant, thus $\mathbf{D}$ is ergodic. Then applying Proposition \ref{Proposition: Schmidt}, it suffices to show that $\mathrm{r}\left(\mathbf{D},\mathrm{Vol}\right)=\mathbb{R}$. Since the defining property of essential values is local in nature, it is enough to verify that on every neighborhood in $M$, every number is an essential value by diffeomorphisms that are supported on this neighborhood. Let $\varphi:U\to\varphi\left(U\right)\subset\mathbb{R}^{d}$ be any local chart for some relatively compact neighborhood $U\subset M$, and denote by $\mathrm{Diff}^{r}\left(U\right)$ those diffeomorphisms $f\in\mathrm{Diff}^{r}\left(M\right)$ for which $\mathrm{Supp}\left(f\right)\subset U$. Denote by $\varphi_{\ast}\mathrm{Vol}$ the pushforward of $\mathrm{Vol}$ along $\varphi$ from $U$ to $\varphi\left(U\right)$. For every $f\in\mathrm{Diff}^{r}\left(\varphi\left(U\right)\right)$, letting $f^{\varphi}:=\varphi^{-1}\circ f\circ\varphi\in\mathrm{Diff}^{r}\left(U\right)$ we have the relation
$$\frac{d\varphi_{\ast}\mathrm{Vol}\circ f}{d\varphi_{\ast}\mathrm{Vol}}\circ\varphi=\frac{d\mathrm{Vol}\circ f^{\varphi}}{d\mathrm{Vol}}.$$
Since $\varphi_{\ast}\mathrm{Vol}$ is mutually absolutely continuous with the Lebesgue measure on $\varphi\left(U\right)$, which we will abbreviate generally by $\lambda$, it follows that 
$$\mathrm{r}\left(\mathrm{Diff}^{r}\left(U\right),\mathrm{Vol}\right)=\mathrm{r}\left(\mathrm{Diff}^{r}\left(\varphi\left(U\right)\right),\varphi_{\ast}\mathrm{Vol}\right)=\mathrm{r}\left(\mathrm{Diff}^{r}\left(\varphi\left(U\right)\right),\lambda\right),$$
where the first ratio set is the one of the tautological action of the group $\mathrm{Diff}^{r}\left(U\right)\subset \mathrm{Diff}^{r}\left(M\right)$ on $U$, and the latter ratio set is the one of the tautological action of the group $\mathrm{Diff}^{r}\left(\varphi\left(U\right)\right)\subset \mathrm{Diff}^{r}\left(\mathbb{R}^{d}\right)$ on $\varphi\left(U\right)$. Thus, it suffices to show that $\mathrm{r}\left(\mathrm{Diff}^{r}\left(\Omega\right),\lambda\right)=\mathbb{R}$ for every open domain $\Omega\subset\mathbb{R}^{d}$ equipped with the Lebesgue measure $\lambda$ and the nonsingular tautological action of the group $\mathrm{Diff}^{r}\left(\Omega\right)$.

Let $\Omega\subset\mathbb{R}^{d}$ be an open domain and let $s\in\mathbb{R}$ be arbitrary. It is typically hard to compute the defining condition of essential value for a general Borel set, so it is a common practice to compute it for sets in a certain generating collection and then, using approximation, show that they are actually essential values (see e.g. \cite[Fact 2.7]{danilenko2023krieger}). In our case of the Euclidean space $\mathbb{R}^{d}$ we work with the collection of open cubes. Thus, in order to show that $s\in\mathrm{r}\left(\mathrm{Diff}^{r}\left(\Omega\right),\lambda\right)$ it suffices to verify the following property.
\begin{quote}
    For every open cube $C\subset\Omega$ and every $\epsilon>0$, there is a Borel set $C_{0}\subset C$ and a diffeomorphism $f_{0}\in\mathrm{Diff}^{r}\left(\Omega\right)$, such that
    $$\lambda\left(C_{0}\right)\geq e^{-\left|s\right|}\lambda\left(C\right),\,f_{0}\left(C_{0}\right)\subset C\text{ and }\log\nabla_{f_{0}}\left(C_{0}\right)\subset\left(s-\epsilon,s+\epsilon\right).$$
\end{quote}
(The cost of considering only cubes is that not only $C_{0}$ is of positive measure, but further the ratio of the volume of $C_{0}$ and of $C$ stays away from zero). Indeed, given an open cube $C=C_{p}\left(u_{0}\right)$ centred at $u_{0}\in\Omega$ with side length $p>0$ (the $\epsilon$ is irrelevant as we will see), consider the open sub-cube 
$$C_{0}=C_{p/\left(e^{\left|s\right|/d}\right)}\left(u_{0}\right)$$
and let $f_{0}\in\mathrm{Diff}^{r}\left(\Omega\right)$ be any diffeomorphism such that 
$$f_{0}\left(u\right)=e^{s/d}\left(u-u_{0}\right)+u_{0}\text{ for every }u\in C_{0}.$$
It is then evident that
$$\lambda\left(C_{0}\right)=\big(p/e^{\left|s\right|/d}\big)^{d}=e^{-\left|s\right|}\lambda\left(C\right),\,f_{0}\left(C_{0}\right)\subset C\text{ and }\log\nabla_{f_{0}}\mid_{C_{0}}\equiv s.\qedhere$$
\end{proof}

\end{appendix}

\bibliographystyle{acm}
\bibliography{References}

\nocite{*}

\end{document}